\documentclass[a4paper,12pt]{amsart}

\usepackage{amssymb}
\usepackage{latexsym}
\usepackage{amsmath}
\usepackage{amscd}
\usepackage{graphicx}
\usepackage{palatino}

\title[Canonical metric on the space of symplectic invariant tensors]
{Canonical metric on the space of symplectic invariant tensors and its applications}

\author{Shigeyuki Morita}
\email{morita@ms.u-tokyo.ac.jp}
\address{Graduate School of Mathematical Sciences, 
University of Tokyo, 
3-8-1 Komaba, 
Meguro-ku, Tokyo, 153-8914, Japan}


\dedicatory{Dedicated to the memory of Professor Akio Hattori}

\newtheorem{thm}{Theorem}[section]
\newtheorem{prop}[thm]{Proposition}
\newtheorem{lem}[thm]{Lemma}
\newtheorem{cor}[thm]{Corollary}
\theoremstyle{definition}
\newtheorem{definition}[thm]{Definition}
\newtheorem{example}[thm]{Example}
\newtheorem{remark}[thm]{Remark}

\newtheorem{conj}[thm]{Conjecture}

\begin{document}

\newcommand{\Mg}{\mathcal{M}_g}
\newcommand{\Mgp}{\mathcal{M}_{g,\ast}}
\newcommand{\Mgb}{\mathcal{M}_{g,1}}

\newcommand{\hg}{\mathfrak{h}_{g,1}}
\newcommand{\ag}{\mathfrak{a}_g}
\newcommand{\Ln}{\mathcal{L}_n}

\newcommand{\Sg}{\Sigma_g}
\newcommand{\Sgb}{\Sigma_{g,1}}
\newcommand{\la}{\lambda}

\newcommand{\Symp}[1]{Sp(2g,\mathbb{#1})}
\newcommand{\symp}[1]{\mathfrak{sp}(2g,\mathbb{#1})}
\newcommand{\gl}[1]{\mathfrak{gl}(n,\mathbb{#1})}

\newcommand{\At}[1]{\mathcal{A}_{#1}^t (H)}
\newcommand{\Hq}{H_{\mathbb{Q}}}

\newcommand{\Ker}{\mathop{\mathrm{Ker}}\nolimits}
\newcommand{\Hom}{\mathop{\mathrm{Hom}}\nolimits}
\renewcommand{\Im}{\mathop{\mathrm{Im}}\nolimits}

\newcommand{\Der}{\mathop{\mathrm{Der}}\nolimits}
\newcommand{\Out}{\mathop{\mathrm{Out}}\nolimits}
\newcommand{\Aut}{\mathop{\mathrm{Aut}}\nolimits}
\newcommand{\Q}{\mathbb{Q}}
\newcommand{\Z}{\mathbb{Z}}
\newcommand{\R}{\mathbb{R}}

\begin{abstract}
Let $\Sigma_g$ be a closed oriented surface of genus $g$
and let $H_\Q$ denote $H_1(\Sigma_g;\Q)$ which we understand
to be the standard symplectic vector space over $\Q$ of dimension $2g$.
We introduce a canonical metric on the space $(H_\Q^{\otimes 2k})^{\mathrm{Sp}}$
of symplectic invariant tensors by analyzing the structure
of the vector space $\Q\mathcal{D}^\ell(2k)$ generated by linear chord diagrams
with $2k$ vertices. 
We decompose
$(H_\Q^{\otimes 2k})^{\mathrm{Sp}}$ as an {\it orthogonal} direct sum of eigenspaces $U_\lambda$ where
$\lambda$ is indexed by 
the set of all the Young diagrams with $k$ boxes. 
This gives a complete
description of how the spaces $(H_\Q^{\otimes 2k})^{\mathrm{Sp}}$ degenerate according 
as the genus decreases
from the stable range $g\geq k$
to the last case $g=1$.

As an application of our canonical metric, we obtain certain relations among the
Mumford-Morita-Miller tautological classes, in a systematic way, which hold in the tautological
algebra {\it in cohomology} of the moduli space of curves. 

\end{abstract}

\renewcommand\baselinestretch{1.1}
\setlength{\baselineskip}{16pt}

\newcounter{fig}
\setcounter{fig}{0}

\maketitle

\section{Introduction and statements of the main results}\label{sec:intro}

Let $\Sigma_g$ be a closed oriented surface of genus $g$.
We denote $H_1(\Sigma_g;\Q)$ simply by $H_\Q$ and let
$$
\mu: H_\Q\otimes H_\Q\rightarrow \Q
$$
be the intersection pairing which is a non-degenerate skew symmetric
bilinear form. If we choose a symplectic basis of $H_\Q$, then 
the automorphism group of $(H_\Q,\mu)$ can be identified with the symplectic 
group $\mathrm{Sp}(2g,\Q)$ and $H_\Q$ serves as the fundamental
representation of $\mathrm{Sp}(2g,\Q)$. Let $(H_\Q^{\otimes{2k}})^{\mathrm{Sp}}$
denote the $\mathrm{Sp}(2g,\Q)$-invariant subspace of the tensor product 
$H_\Q^{\otimes{2k}}$. Now consider the bilinear mapping
$$
\mu^{\otimes 2k}: H_\Q^{\otimes{2k}}\otimes H_\Q^{\otimes{2k}}\rightarrow \Q
$$
defined by
$$
(u_1\otimes\cdots\otimes u_{2k})\otimes (v_1\otimes\cdots\otimes v_{2k})
\mapsto \prod_{i=1}^{2k}\, \mu(u_i,v_i) \quad (u_i, v_i\in H_\Q).
$$
which is clearly symmetric. Hence the restriction of 
$\mu^{\otimes 2k}$ to
the subspace $(H_\Q^{\otimes{2k}})^{\mathrm{Sp}}\otimes (H_\Q^{\otimes{2k}})^{\mathrm{Sp}}$
induces a symmetric bilinear form on $(H_\Q^{\otimes{2k}})^{\mathrm{Sp}}$.

Recall that irreducible representations of the symmetric group
$\mathfrak{S}_k$ is indexed by Young diagrams
$\lambda=[\lambda_1\lambda_2\cdots\lambda_h]$
whose number of boxes, denoted by $|\lambda|=\lambda_1+\cdots\lambda_h$,
is $k$. We denote the corresponding representation by $\lambda_{\mathfrak{S}_k}$.
For each Young diagram $\lambda$ as above, we define two types
of Young diagrams, denoted by $2\lambda$ and $\lambda^\delta$, with $2k$ boxes. One is defined as
$2\lambda=[2\lambda_1 2\lambda_2 \cdots 2\lambda_h]$ and the other is
defined to be $\lambda^\delta=[\lambda_1\lambda_1\lambda_2\lambda_2\cdots\lambda_h\lambda_h]$.
For example, if $\lambda=[431]$, then
$2\lambda=[862]$ and $\lambda^\delta=[443311]$. We call Young
diagrams with the former type (resp. the latter type) 
of {\it even type} (resp. with {\it multiple double floors}).

\begin{thm}

For any $g$, the symmetric bilinear form $\mu^{\otimes 2k}$ on $(H_\Q^{\otimes 2k})^{\mathrm{Sp}}$
is positive definite so that it defines a metric on this space. Furthermore, there exists an orthogonal
direct sum decomposition
$$
(H_\Q^{\otimes 2k})^{\mathrm{Sp}}\cong \bigoplus_{|\lambda |=k,\ h(\lambda)\leq g}
U_{\lambda}
$$
in terms of certain subspaces $U_\lambda$.
With respect to the natural action of $\mathfrak{S}_{2k}$
on $(H_\Q^{\otimes 2k})^{\mathrm{Sp}}$, each subspace $U_\lambda$
is an irreducible $\mathfrak{S}_{2k}$-submodule and we have an isomorphism
$$
U_\lambda\cong (\lambda^\delta)_{\mathfrak{S}_{2k}}.
$$
In other words, the above orthogonal
direct sum decomposition gives also the irreducible decomposition of the
$\mathfrak{S}_{2k}$-module $(H_\Q^{\otimes 2k})^{\mathrm{Sp}}$.

\label{thm:ortho}
\end{thm}

\begin{cor}
The dimension of the space of invariant tensors $(H_\Q^{\otimes 2k})^{\mathrm{Sp}}$
is given by
$$
\dim\, (H_\Q^{\otimes 2k})^{\mathrm{Sp}}
=\sum_{ |\lambda|=k,\ h(\lambda)\leq g} \mathrm{dim}\, (\lambda^\delta)_{\mathfrak{S}_{2k}}.
$$
\label{cor:dim}
\end{cor}

\begin{remark}
Recall here that there are a number of explicit formulas for the dimension of irreducible representations of
symmetric groups, e.g. the hook length formula, the number of standard tableaux on the
Young diagram and the formula of Frobenius (cf \cite{fh}). In particular
$$
\dim\, U_{[k]}=\dim\,[k^2]_{\mathfrak{S}_{2k}}=\frac{1}{k+1}\binom{2k}{k}\quad (\text{the $k$-th Catalan number})
$$
is equal to $\dim\, (H_\Q^{\otimes 2k})^{\mathrm{Sp}}$ for the case $g=1$.
We mention that, in this case,  a basis of this space was given
in a classical paper \cite{RTW} by Rumer, Teller and Weyl.
Furthermore Mihailovs \cite{mihailovs} extended this work to obtain a nice
basis of $(H_\Q^{\otimes 2k})^{\mathrm{Sp}}$
for any $g$.
However, their basis is different from ours because it
depends essentially on the choice of ordering of tensors so that it is not canonical.
\end{remark}

A {\it linear chord diagram} with $2k$ vertices is a partition of the set
$\{1,2,\cdots,2k\}$ into $k$-tuple 
$$
C=\{\{i_1,j_1\},\cdots,\{i_k,j_k\}\}\quad 
$$
of pairs $\{i_\ell,j_\ell\}\ (\ell=1,\cdots,k)$. Here we assume 
$$
i_1<\cdots<i_k, \ i_\ell<j_\ell\ (\ell=1,\cdots,k)
$$
so that the above expression is uniquely determined.
We consider $C$ to be a graph with $2k$ vertices and $k$ edges (chords)
each of which connects $i_\ell$ with $j_\ell$ for $\ell=1,\cdots,k$.
Let
$
\mathcal{D}^{\ell}(2k)
$
denote the set of all the linear chord diagrams with $2k$ vertices.
It has $(2k-1)!!$ elements and let $\Q\mathcal{D}^{\ell}(2k)$ be the vector
space over $\Q$ spanned by $\mathcal{D}^{\ell}(2k)$.
The symmetric group $\mathfrak{S}_{2k}$ acts on $\mathcal{D}^{\ell}(2k)$ naturally 
so that $\Q\mathcal{D}^{\ell}(2k)$ has a structure of an $\mathfrak{S}_{2k}$-module.

We define a symmetric bilinear mapping
$$
\langle\ \, ,\ \rangle: \Q\mathcal{D}^{\ell}(2k)\times \Q\mathcal{D}^{\ell}(2k)\rightarrow\Q[g]
$$
by setting
$$
\langle C, C' \rangle= (-1)^{k-r} (2g)^r
$$
for any two elements $C, C'\in \mathcal{D}^\ell(2k)$, where $r$ denotes the number
of connected components of the graph $C\cup C'$
(here we consider interiors of $2k$ edges to be mutually {\it disjoint}), and extending this bilinearly. 
If we let $g$  to take a specific value, then the above can be understood 
to be a symmetric bilinear pairing. It is clear that this pairing
is $\mathfrak{S}_{2k}$-invariant. Namely the equality
$$
\langle \gamma(C), \gamma(C')\rangle=\langle C, C'\rangle\quad (C,C'\in \mathcal{D}^{\ell}(2k))
$$
holds for any $\gamma\in\mathfrak{S}_{2k}$. Let $M_{2k}$ denote
the corresponding intersection matrix. In other words, it is the matrix
which represents the linear mapping 
$$
\mathcal{M}: \Q\mathcal{D}^{\ell}(2k)\rightarrow\Q\mathcal{D}^{\ell}(2k)
$$
defined as 
$$
\mathcal{M}(C)=\sum_{D\in\mathcal{D}^\ell(2k)} \langle C,D\rangle \, D.
$$
\begin{thm}
The intersection matrix $M_{2k}$ is positive semi-definite for any $g$,
so that it defines a metric on
$\Q\mathcal{D}^{\ell}(2k)$ depending on $g$.
Furthermore, there exists an orthogonal
direct sum decomposition
$$
\Q\mathcal{D}^{\ell}(2k)\cong \bigoplus_{|\lambda |=k}
E_{\lambda}
$$
in terms of eigenspaces $E_\lambda$ of $\mathcal{M}$ which
are defined independently of the genus $g$.
The eigenvalue of $E_\lambda$ , denoted by $\mu_\lambda$, is given by
the following formula
$$
\mu_\lambda=\prod_{\text{b: box of $\lambda$}} (2g-2s_b+t_b)
$$
where $s_b$ denotes the number of columns of $\lambda$ 
which are on the left of the column containing $b$ and  
$t_b$ denotes the number of rows which are above the 
row containing $b$.

With respect to the natural action of $\mathfrak{S}_{2k}$
on $\Q\mathcal{D}^{\ell}(2k)$, $E_\lambda$
is an irreducible $\mathfrak{S}_{2k}$-submodule and we have an isomorphism
$$
E_\lambda\cong (2\lambda)_{\mathfrak{S}_{2k}}.
$$
In other words, the above orthogonal
direct sum decomposition gives also the irreducible decomposition of the
$\mathfrak{S}_{2k}$-module $\Q\mathcal{D}^{\ell}(2k)$.

\label{thm:lcd}
\end{thm}

\begin{remark}
It turns out that the main ingredient of the above theorem was already proved by Hanlon and Wales \cite{hw}
in a broader context of Brauer's
centralizer algebras. More precisely, they determined the eigenspaces and eigenvalues together with
the irreducible decomposition of $\Q\mathcal{D}^{\ell}(2k)$ as an
$\mathfrak{S}_{2k}$-module. 
However the properties of semi-positivity and orthogonality were not shown
partly because their sign convention is different from ours
so that the intersection matrix is not the same.
Here we would like to present our own proof for this part also 
which may hopefully be more geometric than
their original proof.
\end{remark}

\begin{thm}
The correspondence
$$
\lambda \mapsto \mu_\lambda
$$
defines an injective mapping from the set of all the Young diagrams
$\lambda$ with $k$ boxes into the space of polynomials in $g$ of degree $k$.
\end{thm}

\begin{example}
The largest (resp. the smallest) eigenvalue occurs for $\lambda=[1^k]$
(resp. $\lambda=[k]$) and the corresponding eigenvalues are
$$
\mu_{[1^k]}=2g(2g+1)\cdots (2g+k-1), \quad
\mu_{[k]}=2g(2g-2)\cdots (2g-2k+2).
$$
\end{example}

The space $\Q\mathcal{D}^\ell (2k)$ is a universal model for
the symplectic invariant tensors because we can make the following construction.
First there exists a natural homomorphism
$$
\Phi: \Q\mathcal{D}^\ell (2k) \rightarrow (H_\Q^{\otimes 2k})^{\mathrm{Sp}}
$$
which is surjective for any $g$ and ``anti" $\mathfrak{S}_{2k}$-equivariant
(see Proposition \ref{prop:anti} for details).
In the dual setting, there exists also a natural homomorphism
$$
\mathcal{K}: H_\Q^{\otimes 2k}\rightarrow \Q\mathcal{D}^\ell (2k)
$$
which enumerate all the contractions and can detect $\mathrm{Sp}$-invariant
component of any given tensor
(see section $\S3$, in particular Proposition \ref{prop:dual}, for details). 
We use the same symbol $\mathcal{K}$ for its restriction
to the subspace $(H_\Q^{\otimes 2k})^{\mathrm{Sp}}$.

\begin{thm}
The following diagram is commutative
$$
\begin{CD}
\Q\mathcal{D}^\ell (2k)   @>{\mathcal{M}}>> \Q\mathcal{D}^\ell (2k) \\
@V{\Phi}VV @|\\
(H_\Q^{\otimes 2k})^{\mathrm{Sp}} @>{\mathcal{K}}>> \Q\mathcal{D}^\ell (2k)
\end{CD}
$$
and all the homomorphisms $\Phi, \mathcal{M}, \mathcal{K}$ are isomorphisms for $g\geq k$.
For any $\lambda$ with $|\lambda|=k$, $U_{\lambda}=\Phi(E_{\lambda'})$
where $\lambda'$ denotes the conjugate
Young diagram of $\lambda$.

Furthermore the subspaces $\mathcal{K}(U_\lambda)=E_{\lambda'}\subset \Q \mathcal{D}^\ell(2k)$ are mutually
orthogonal to each other with respect to 
the ordinary Euclidean metric on $\Q \mathcal{D}^\ell(2k)$ induced by taking 
$\mathcal{D}^\ell(2k)$ as an orthonormal basis. 
\label{thm:aalpha}
\end{thm}

This work was originally motivated by an attempt, begun some years ago, to extend
our former work \cite{morita03} on the structure of
the tautological algebra of the moduli space of curves. 
We present in $\S6$ 
an explicit set of relations among the Mumford-Morita-Miller
tautological classes by making use of the canonical metric introduced in this
paper (see Theorem \ref{thm:rel} and Conjecture \ref{conj:complete}). 
In a forthcoming joint paper \cite{mss5} with Sakasai and Suzuki,
which is a sequel to this paper,
we show a close relation between the structure of the tautological algebra 
of moduli space of curves and stability and
degeneration of certain plethysm of representations.

{\it Acknowledgement} The author would like to thank Carel Faber,
Richard Hain and Nariya Kawazumi for enlightening discussions
about the cohomology of the mapping class groups
on various occasions in the $1990$'s and early $2000$'s.
He also would like to thank Dan Petersen for informing him
about a paper \cite{hw} by Hanlon and Wales.
Thanks are also due to
Takuya Sakasai and Masaaki Suzuki
for their helps in recent years
which greatly encouraged the author to write up this paper.

\section{Space of linear chord diagrams and $(H^{\otimes 2k})^{\mathrm{Sp}}$}

For each linear chord diagram $C\in \mathcal{D}^\ell(2k)$, define
$$
a_C\in (H_\Q^{\otimes 2k})^{\mathrm{Sp}}
$$
by permuting the elements $(\omega_0)^{\otimes 2k}$ in such a way that the $s$-th part
$(\omega_0)_s$ of this tensor product goes to $(H_\Q)_{i_s}\otimes (H_\Q)_{j_s}$, where $(H_\Q)_{i}$
denotes the $i$-th component of $H_\Q^{\otimes 2k}$, and multiplied by the factor
$$
\mathrm{sgn}\, C=\mathrm{sgn}
\begin{pmatrix}
1 & 2 & \cdots & 2k-1 & 2k\\
i_1 & j_1 & \cdots & i_k & j_k
\end{pmatrix}
.
$$
Here $\omega_0\in (H_\Q^{\otimes 2})^{\mathrm{Sp}}$ denotes the symplectic class.

Now define a linear mapping
$$
\Phi: \Q\mathcal{D}^\ell (2k) \rightarrow (H_\Q^{\otimes 2k})^{\mathrm{Sp}}
$$
by setting $\Phi(C)=a_C$.

\begin{prop}
The correspondence 
$$
\Phi: \Q\mathcal{D}^\ell (2k) \rightarrow (H_\Q^{\otimes 2k})^{\mathrm{Sp}}
$$
is surjective for any $g$ and bijective for any $g\geq k$. Furthermore this correspondence
is ``anti" $\mathfrak{S}_{2k}$-equivariant in the sense that
$$
\Phi(\gamma (C))=\mathrm{sgn}\,\gamma\  \gamma(\Phi(C))
$$
for any $C\in \Q\mathcal{D}^\ell (2k)$ and $\gamma\in\mathfrak{S}_{2k}$.
\label{prop:anti}
\end{prop}

\begin{proof}
The surjectivity follows from a classical result of Weyl on symplectic invariants.
The latter part can be shown by comparing  how the signs of permutations act
on the symplectic invariant tensors and the linear chord diagrams.
\end{proof}

\begin{prop}
The following diagram is commutative
$$
\begin{CD}
\Q\mathcal{D}^\ell (2k) \otimes \Q\mathcal{D}^\ell (2k)  @>{\langle\ ,\ \rangle}>> \Q[g]\\
@V{\Phi\otimes\Phi}VV @VV{\text{evaluation}}V\\
(H_\Q^{\otimes 2k})^{\mathrm{Sp}}\otimes (H_\Q^{\otimes 2k})^{\mathrm{Sp}} @>{\mu^{\otimes 2k}}>> \Q.
\end{CD}
$$
\label{prop:lcdmu}
\end{prop}

To prove this, we prepare a few facts.

\begin{lem}
Let $x_1,\cdots,x_g,y_1,\cdots,y_g$ be a symplectic basis of $H_\Q$ and let
$$
\omega_0=\sum_{i=1}^g \{x_i\otimes y_i-y_i\otimes x_i\}\in (H_\Q^{\otimes 2})^{\mathrm{Sp}}
$$
be the symplectic class.

$\mathrm{(i)}\  \text{For any two elements $u,v\in H_\Q$, we have  the identity}$
$$
\mu(u,v)=\sum_{i=1}^g \{\mu(u,x_i)\mu(v,y_i)-\mu(u,y_i)\mu(v,x_i)\}.
$$ 

$\mathrm{(ii)}\  \text{For any two elements $u,v\in H_\Q$, we have  the identity}$
$$
\mu(u,v)=\mu^{\otimes 2}(u\otimes v,\omega_0).
$$ 
\label{lem:muom}
\end{lem}

\begin{proof}
It is easy to verify (i) while (ii) is just a restatement of (i) in terms of the linear mapping
$$
\mu^{\otimes 2}: H_\Q^{\otimes 2}\otimes H_\Q^{\otimes 2}\rightarrow \Q.
$$
\end{proof}

\begin{lem}[Lemma 3.3 of \cite{morita03}]
For any two indices $i, j$ with $1\leq i < j \leq 2k$, let
$$
p_{ij}:H_\Q^{\otimes 2k} \rightarrow H_\Q^{\otimes 2k}
$$
be the linear map defined by first taking the intersection number of the $i$-th entry with
the $j$-th entry and then putting the element $\omega_0$ there. Namely we set
\begin{align*}
p_{ij}(u_1\otimes\cdots \otimes u_i\otimes \cdots& \otimes u_j\otimes \cdots\otimes u_{2k})=\\
\mu(u_i,u_j) \sum_{s=1}^g& \{u_1\otimes\cdots \otimes x_s\otimes \cdots \otimes y_s\otimes \cdots\otimes u_{2k})\\
&- u_1\otimes\cdots \otimes y_s\otimes \cdots \otimes x_s\otimes \cdots\otimes u_{2k}\}.
\end{align*}
Then we have
$$
p_{ij}(a_C)=
\begin{cases}
2g\ a_C \quad \{i,j\}\in C\\
- a_{C'} \quad \{i,j\}\notin C
\end{cases}
$$
where $C'\in\mathcal{D}^\ell (2k)$ is the linear chord diagram defined as follows. 
Let $i', j'$ be the
indices such that $\{i, j'\}, \{i', j\}\in C$. Then
$$
C'=C\setminus \{\{i, j'\}, \{i', j\}\} \cup \{\{i, j\}, \{i', j'\}\}.
$$
\label{lem:qij}
\end{lem}

\begin{proof}[Proof of Proposition $\ref{prop:lcdmu}$]
By the linearity, it is enough to prove the following. For any two 
linear chord diagrams $C,C'\in \mathcal{D}^\ell(2k)$, we have the identity
\begin{equation}
\mu^{\otimes 2k}(a_C\otimes a_{C'})=\langle C, C'\rangle =(-1)^{k-r} (2g)^r
\label{eq:id}
\end{equation}
where the right equality is the definition of the pairing 
$\langle\ , \ \rangle$ and $r$ denotes the number of connected components of the graph $C\cup C'$.

To prove the above equality \eqref{eq:id}, we decompose $\mu^{\otimes 2k}$ as
$$
\mu^{\otimes 2k}=\prod_{\ell=1}^k \mu_{i_\ell,j_\ell}^{\otimes 2}
$$
where $C=\{\{i_1,j_1\},\cdots,\{i_k,j_k\}\}$ as before and $\mu_{i_\ell,j_\ell}^{\otimes 2}(u\otimes v)$ 
$(u, v\in H_\Q^{\otimes 2k})$ denotes
taking the product of $\mu(u_{i_\ell},v_{i_\ell})$ and $\mu(u_{j_\ell},v_{j_\ell})$.
Now we consider the first chord $\{i_1,j_1\}$ of $C$ 
(actually $i_1=1$) and apply $\mu_{i_1,j_1}^{\otimes 2}$  on $a_C\otimes a_{C'}$. By Lemma \ref{lem:muom}
this can be calculated by applying the operator $p_{i_1j_1}$ on $a_{C'}$ and by Lemma \ref{lem:qij} the answer is
given as follows. Namely if $\{i_1,j_1\}\in C'$, then we obtain the number $\mu^{\otimes 2}(\omega_0\otimes\omega_0)=2g$
and the linear chord diagram $C'$ is unchanged so that $C'_1=C'$
while if $\{i_1,j_1\}\notin C'$, then we obtain the factor $-1$ and the new linear chord diagram $C'_1$ is 
the one obtained from $C'$ by replacing $\{i_1,j'_1\}, \{i'_1,j_1\}\in C'$ by $\{i_1,j_1\}, \{i'_1,j'_1\}$.
In the latter case, it is easy to observe that the number of connected components of the graph $C\cup C'_1$ 
increases by one than that of the graph $C\cup C'$ because one of the connected components of 
the latter graph is split into two components one of which corresponds to the common edge $\{i_1,j_1\}$
of the two graphs $C$ and $C'_1$. Next we consider the second chord $\{i_2,j_2\}$ of $C$ 
and apply $\mu_{i_2,j_2}^{\otimes 2}$  on $a_C\otimes a_{C'_1}$. By the same argument as above,
if $\{i_2,j_2\}\in C'_1$, then we obtain the number $2g$
and the linear chord diagram $C'_1$ is unchanged so that $C'_2=C'_1$
while if $\{i_2,j_2\}\notin C'_1$, then we obtain the factor $-1$ and the new linear chord diagram $C'_2$ is 
the one obtained from $C'_1$ by replacing $\{i_2,j'_2\}, \{i'_2,j_2\}\in C'_1$ by $\{i_2,j_2\}, \{i'_2,j'_2\}$.
In the latter case, the number of connected components of the graph $C\cup C'_2$ 
increases by one than that of the graph $C\cup C'_1$ because one of the connected components of 
the latter graph is split into two components one of which corresponds to the common edge $\{i_2,j_2\}$
of the two graphs $C$ and $C'_2$. Continuing in this way, we finally arrive at the last one $C'_k$ which 
coincides with $C$ and along the way we obtain the factor $(2g)$ $r$-times while the factor $(-1)$ 
$(k-r)$ times for the evaluation of $\mu^{\otimes 2k}$. This completes the proof.
\end{proof}

\begin{prop}
The irreducible decomposition of the $\mathfrak{S}_{2k}$-module $\Q\mathcal{D}^\ell (2k)$
is given by
$$
\Q\mathcal{D}^\ell (2k)=\bigoplus_{|\lambda|=k} E_\lambda
$$
where $E_\lambda$ is isomorphic to $(2\lambda)_{\mathfrak{S}_{2k}}$.
\label{prop:E}
\end{prop}

\begin{proof}
The symmetric group $\mathfrak{S}_{2k}$ acts on $\mathcal{D}^\ell (2k)$ naturally and
this action is clearly transitive. If we denote by $C_0\in \mathcal{D}^\ell (2k)$ the 
linear chord diagram 
$$
C_0=\{\{1,2\},\{3,4\},\cdots, \{2k-1,2k\}\}
$$
then the isotropy group of it
is the subgroup $H_k\subset \mathfrak{S}_{2k}$ 
consisting of permutations which fix  $C_0$ considered as an element
of the power set of $\{1,2,\cdots, 2k\}$. $H_k$ is isomorphic to 
$$
(\mathfrak{S}_2\times \overset{\text{$k$-times}}{\cdots}\times \mathfrak{S}_2)\rtimes\mathfrak{S}_{k}
$$
and we have a bijection
\begin{equation}
\mathfrak{S}_{2k}/H_k \cong \mathcal{D}^\ell (2k)
\label{eq:scd}
\end{equation}
from the set of left $H_k$-cosets of $\mathfrak{S}_{2k}$ to $\mathcal{D}^\ell (2k)$.
Note that the number of elements of both sets is equal to $(2k-1)!!$. It is easy to see that the
above bijection \eqref{eq:scd} is an isomorphism of $\mathfrak{S}_{2k}$-set.
As is well known, the $\mathfrak{S}_{2k}$-module $\mathfrak{S}_{2k}/H_k$
is isomorphic to the induced representation $\mathrm{Ind}^{\mathfrak{S}_{2k}}_{H_k} 1_{H_k}$,
where $1_{H_k}$ denotes the one dimensional trivial representation of $H_k$.
Hence the proof of the present proposition is reduced to the following one.
\end{proof}

\begin{prop}
Let $\mathfrak{S}_{2k}$ be the symmetric group of order $2k$ and let $H_k$ denote the subgroup
defined above. Then the irreducible decomposition of the induced representation $\mathrm{Ind}^{\mathfrak{S}_{2k}}_{H_k} 1_{H_k}$
is given by
$$
\mathrm{Ind}^{\mathfrak{S}_{2k}}_{H_k} 1_{H_k}\cong\bigoplus_{|\lambda|=k} (2\lambda)_{\mathfrak{S}_{2k}}.
$$
\label{prop:hk}
\end{prop}
This is an elementary fact about representations of symmetric groups 
and can be proved by applying standard argument in the theory of representations of finite groups
(see e.g. \cite{iwahori}).
However, in order to motivate later arguments and also for completeness, 
we give a proof along the line of the ingredient of this paper.

We first recall a few terminologies about representations of the symmetric groups from \cite{fh}. 
We will use them also in section $\S5$.

For a Young diagram $\lambda$ with $|\lambda|=d$, choose a Young tableau $T$ (say a standard one) on $\lambda$.
Then we set
\begin{align*}
P_\lambda&=\{\gamma\in \mathfrak{S}_d; \text{$\gamma$ preserves each row}\}\\
Q_\lambda&=\{\gamma\in \mathfrak{S}_d; \text{$\gamma$ preserves each column}\}
\end{align*}
$$
a_\lambda=\sum_{\gamma\in P_\lambda} \gamma,\quad b_\lambda=\sum_{\gamma\in Q_\lambda} \mathrm{sgn}\,\gamma\,\gamma
$$
The element
$$
c_\lambda=a_\lambda b_\lambda\in \Q[\mathfrak{S}_d]
$$
is called a Young symmetrizer. We can also use the following element
$$
c'_\lambda=b_\lambda a_\lambda \in \Q[\mathfrak{S}_d]
$$
which is another form of Young symmetrizer and can play the same role as above.
In this paper, we use both ones.

Next we recall the following well known result.

\begin{thm}[Mackey's restriction formula, see e.g. \cite{serre}]
Let $G$ be a finite group and let $H, K$ be subgroups of $G$. Let
$\rho$ be a representation of $H$ and for each $(H,K)$-double coset 
representative $g$, let $\rho_g$ be the representation of $H_g=gHg^{-1}\cap K$
defined by $\rho_g(x)=\rho(g^{-1}x g)\ (x\in gHg^{-1})$.
Then we have the equality
$$
\mathrm{Res}^G_K\, \mathrm{Ind}^G_H\, \rho=\bigoplus_{g\in K\backslash G/H} \mathrm{Ind}^K_{H_g}\, \rho_g.
$$
\label{thm:mackey}
\end{thm}

\begin{proof}[Proof of Proposition $\ref{prop:hk}$]
First, we show that 
\begin{equation}
\left|H_k\backslash \mathfrak{S}_{2k}/H_k\right|=p(k).
\label{eq:pk}
\end{equation}
Namely the number of elements of the double cosets 
$H_k\backslash \mathfrak{S}_{2k}/H_k$ is equal to $p(k)$ which is the
number of partitions of $k$. Recall that the set of left $H_k$-cosets of $\mathfrak{S}_{2k}$
can be canonically identified with $\mathcal{D}^\ell(2k)$. $H_k$ acts on this set
from the left and it is easy to see that the space of orbits of this action can be
identified with the set of $H_k$-double cosets of $\mathfrak{S}_{2k}$.
Now define the {\it $C_0$-relative type} of an element $C\in\mathcal{D}^\ell(2k)$,
denoted by $\tau_{C_0}(C)$,
as follows. Consider the graph $\Gamma_C=C_0\cup C$. Each connected component
of $\Gamma_C$ contains even number of vertices from the set of all vertices $\{1,2,\cdots,2k\}$.
If we enumerate all these even numbers in the decreasing order, then we obtain
an even type Young diagram $2\lambda$ where $\lambda$ is
a partition of $k$ (equivalently a Young diagram with $k$ boxes). 
We put $\tau_{C_0}(C)=\lambda$. For example $\tau_{C_0}(C_0)=[1^k]$.
Then we can show the following. For any two linear chord diagrams $C,D\in\mathcal{D}^\ell(2k)$,
the corresponding left $H_k$-cosets belong to the same right $H_k$-orbits if and only if
$\tau_{C_0}(C)=\tau_{C_0}(D)$. Also, for any Young diagram $\lambda$ with $|\lambda|=k$,
there exists $C\in \mathcal{D}^\ell(2k)$ such that $\tau_{C_0}(C)=\lambda$.
Since there are exactly $p(k)$ Young diagrams $\lambda$ with $|\lambda|=k$, the claim 
\eqref{eq:pk} is proved.

Next, by the Frobenius reciprocity theorem, for any Young diagram $\mu$ with $|\mu|=2k$ we have
\begin{equation}
\text{multiplicity of $\mu_{\mathfrak{S}_{2k}}$ in $\mathrm{Ind}^{\mathfrak{S}_{2k}}_{H_k} 1_{H_k}$}
=\mathrm{dim}\,(\mathrm{Res}^{\mathfrak{S}_{2k}}_{H_k}\, \mu_{\mathfrak{S}_{2k}})^{H_k}.
\label{eq:mult}
\end{equation}
We show that for any Young diagram $\lambda$ with $|\lambda|=k$
\begin{equation}
\mathrm{dim}\,(\mathrm{Res}^{\mathfrak{S}_{2k}}_{H_k}\, (2\lambda)_{\mathfrak{S}_{2k}})^{H_k}\geq 1.
\label{eq:2l}
\end{equation}
To prove this, it suffices to show that 
$$
g_\lambda=\left(\sum_{\gamma\in H_k} \gamma\right) c_{(2\lambda)}\not=0\ \in \Q[\mathfrak{S}_{2k}]
$$
because this is clearly an $H_k$-invariant element. 
Here we choose a standard tableau $T$ on $2\lambda$ to be the one defined as follows.
We give the numbers $1,2,\cdots,2\lambda_1$ to fill the first row from the left to the right
and then the numbers $2\lambda_1+1,\cdots,2\lambda_1+2\lambda_2$ to the second row
from the left to the right
and so on, where $\lambda=[\lambda_1\cdots\lambda_h]$. Consider
$k$ ordered pairs 
$$
\langle1,2\rangle, \langle2,3\rangle,\cdots, \langle2k-1, 2k\rangle
$$
of numbers each of which is placed on some row of $2\lambda$. Now the subgroup
$H_k$ is isomorphic to the semi-direct product $(\mathfrak{S}_2)^k\rtimes\mathfrak{S}_k$
and the $i$-th $\mathfrak{S}_2$ acts on $\langle 2i-1,2i\rangle$ naturally while $\mathfrak{S}_k$
acts on the above pairs by permutations (preserving each order). It follows that we can write
$$
\sum_{\gamma\in H_k}=\sum_{\gamma\in \mathfrak{S}_k} \gamma \sum_{\delta\in (\mathfrak{S}_2)^k} \delta
$$
and also that $(\mathfrak{S}_2)^k$ is contained in $P_{2\lambda}$ which is the subgroup of 
$\mathfrak{S}_{2k}$ consisting of elements which permute the numberings
but within each row of $2\lambda$.
Therefore
$$
g_\lambda=\left(\sum_{\gamma\in \mathfrak{S}_k} \gamma\right) \left(\sum_{\delta\in (\mathfrak{S}_2)^k} \delta\right)
a_{2\lambda}b_{2\lambda}
=2^k \left(\sum_{\gamma\in \mathfrak{S}_k} \gamma\right) c_{2\lambda}
$$
because $\delta\, a_{2\lambda}=a_{2\lambda}$ for any $\delta\in P_{2\lambda}$.
Recall that the coefficient of the unit element $1\in \mathfrak{S}_{2k}$ in the expression $c_{2\lambda}$ is $1$.
To prove the assertion $g_\lambda\not=0$, it suffices to show that the coefficient of $1$ is different from $0$.
Suppose for some elements $\gamma\in \mathfrak{S}_k, g\in P_{2\lambda}, h\in Q_{2\lambda}$, the
equality
\begin{equation}
\gamma g h=1
\label{eq:gh}
\end{equation}
holds. It implies that $gh\in \mathfrak{S}_k$. Let $\langle 2i-1,2i\rangle$ be any one of the above ordered
pair. Then we can write
$$
gh \langle 2i-1,2i\rangle=\langle gh(2i-1),gh(2i)\rangle=\langle 2j-1,2j\rangle
$$
for some $j$, because the action of $gh=\gamma^{-1}\in \mathfrak{S}_{2k}$ reduces to permutations of the ordered pairs preserving 
the order of each pair. Recall that $h(2i-1)$ (resp. $h(2i))$ lies in the same column as $2j-1$ (resp. $2j$).
Since $g$ preserves each row and still $gh(2i-1)=2j-1$ and $gh(2i)=2j$ belong to the same row, we can conclude that
$h(2j-1)$ and $h(2j)$ belong to the same row. It follows that the image of $\langle 2i-1,2i\rangle$ under
$h$ is again one of the ordered pairs. Since $\langle 2i-1,2i\rangle$ was an arbitrary ordered pair,
we can conclude that $h$ is contained in the subgroup $\mathfrak{S}_k\subset\mathfrak{S}_{2k}$.
Now observe that the $sgn$ of any element of this subgroup $\mathfrak{S}_k$, but considered as
an element of $\mathfrak{S}_{2k}$, is $+1$ because it acts on $\{1,2,\cdots,2k\}$
by an even permutation. Thus, whenever the equality \eqref{eq:gh} occurs, we have $\mathrm{sgn}\, h=1$.
We can now conclude that the coefficient of $1$ in $g_\lambda$ is different from $0$ as claimed.

Finally, we prove
\begin{equation}
\mathrm{dim}\, \left(\mathrm{Res}^{\mathfrak{S}_{2k}}_{H_k}\, \mathrm{Ind}^{\mathfrak{S}_{2k}}_{H_k} 1_{H_k}\right)^{H_k}=p(k).
\label{eq:gpk}
\end{equation}
It is easy to see that the representation 
$\mathrm{Res}^{\mathfrak{S}_{2k}}_{H_k}\, \mathrm{Ind}^{\mathfrak{S}_{2k}}_{H_k} 1_{H_k}$
is nothing other than the one associated with the left $H_k$-action on the
set $\mathcal{D}^\ell(2k)\cong \mathfrak{S}_{2k}/H_k$. Therefore, by a well-known
basic fact, the multiplicity of the trivial representation occurring in it
is the same as the number of orbits of this action which is $p(k)$ by
\eqref{eq:pk}.
More precisely, we can apply Theorem \ref{thm:mackey}
to the most simple case where $G=\mathfrak{S}_{2k}, H=K=H_k$ and $\rho$ is the
one dimensional trivial representation. Then we obtain
$$
\mathrm{Res}^{\mathfrak{S}_{2k}}_{H_k}\, \mathrm{Ind}^{\mathfrak{S}_{2k}}_{H_k} 1_{H_k}
=\bigoplus_{g\in H_k\backslash {\mathfrak{S}_{2k}}/H_k} \mathrm{Ind}^{H_k}_{(H_k)_g}\, 1_{(H_k)_g}.
$$
Taking the $H_k$-trivial part, this implies
$$
\left(\mathrm{Res}^{\mathfrak{S}_{2k}}_{H_k}\, \mathrm{Ind}^{\mathfrak{S}_{2k}}_{H_k} 1_{H_k}\right)^{H_k}
=\bigoplus_{g\in H_k\backslash {\mathfrak{S}_{2k}}/H_k} \left(\mathrm{Ind}^{H_k}_{(H_k)_g}\, 1_{(H_k)_g}\right)^{H_k}.
$$
Again by the Frobenius reciprocity theorem, we have
$$
\text{$\mathrm{dim}\, \left(\mathrm{Ind}^{H_k}_{(H_k)_g}\, 1_{(H_k)_g}\right)^{H_k}=1\ $ for any $g\in H_k\backslash {\mathfrak{S}_{2k}}/H_k$}
$$
and the claim \eqref{eq:gpk} follows.

Since there are exactly $p(k)$ Young diagrams $\lambda$ with $|\lambda|=k$, \eqref{eq:2l} implies 
\begin{equation*}
\mathrm{dim}\, (\mathrm{Res}^{\mathfrak{S}_{2k}}_{H_k}\, \mu_{\mathfrak{S}_{2k}})^{H_k}
=\begin{cases}
1\quad \text{if $\mu=2\lambda$ for some $\lambda$}\\
0\quad \text{otherwise}.
\end{cases}
\label{eq:n2l}
\end{equation*}
In view of \eqref{eq:mult}, we can now conclude
$$
\mathrm{Ind}^{\mathfrak{S}_{2k}}_{H_k} 1_{H_k}\cong\bigoplus_{|\lambda|=k} (2\lambda)_{\mathfrak{S}_{2k}}
$$
as required.

\end{proof}

\begin{prop}
A basis for the subspace $E_\lambda$ is given by
$$
\{c_\tau (C_0); \tau\in ST(\lambda)\}
$$
where $ST(\lambda)$ denotes the set of all the standard tableaux on the Young
diagram $2\lambda$ and $c_\tau$ denotes the Young symmetrizer corresponding to
$\tau\in ST(\lambda)$.
\end{prop}

\begin{prop}
The correspondence
$$
\lambda \mapsto \mu_\lambda
$$
defines an injective mapping from the set of all the Young diagrams
$\lambda$ with $k$ boxes into the space of polynomials in $g$ of degree $k$.
\label{prop:lp}
\end{prop}

\begin{proof}
It is enough to recover the shape of any Young diagram $\lambda=[\lambda_1\cdots \lambda_h]$ from the 
corresponding polynomial $\mu_\lambda$.
Recall that $\mu_\lambda$ is the product of $\mu_b=2g-2s_b+t_b$ where $b$ runs through all the boxes of
$\lambda$. It is clear that the largest (reap. the smallest) one among $\mu_b$ is $(2g+h-1)$ (resp. $(2g-2\lambda_1+2)$).
It follows that $\mu_\lambda$ determines $h$, which is the number of rows of $\lambda$, and $\lambda_1$.
Now consider the quotient
$$
\mu_\lambda^{(1)}=\frac{\mu_\lambda}{2g(2g+1)\cdots (2g+h-1)(2g-2)\cdots (2g-2\lambda_1+2)}
$$
which is again a polynomial in $g$ of degree $(k-h-\lambda_1+1)$. Let $\lambda^{(1)}$ be the
Young diagram obtained from $\lambda$ by deleting the first row and the first column. It has
$(k-h-\lambda_1+1)$ boxes. For example, if $\lambda=[432]$, then $\lambda^{(1)}=[21]$.
Now $\mu_\lambda^{(1)}$ is expressed as the product of $\mu_b$ where $b$ runs through
all the boxes of $\lambda^{(1)}$ which we consider as a subset of the original $\lambda$.
Now it is easy to see that the largest (reap. the smallest) one among 
$\mu_b\ (b\in \lambda^{(1)}$ is $(2g+h_2-3)$ (resp. $(2g-2\lambda_2+3)$)
where $h_2$ denotes the number of boxes of the second column of $\lambda$.
It follows that we can recover $h_2$ and $\lambda_2$. Continuing this kind of argument,
we see that we can eventually recover the shape of $\lambda$. This completes the proof.
\end{proof}

\section{Orthogonal decomposition of $(H_\Q^{\otimes 2k})^{\mathrm{Sp}}$}\label{sec:orth}

In this section, we deduce Theorem \ref{thm:ortho} from Theorem \ref{thm:lcd} whose proof
will be given in the next section. See Table \ref{tab:8} for an example of the orthogonal direct
sum decomposition.
\begin{table}[h]
\caption{$\text{Orthogonal decomposition of $(H_\Q^{\otimes 8})^{\mathrm{Sp}}$\  }$}
\begin{center}
\begin{tabular}{|c|c|c|l|}
\noalign{\hrule height0.8pt}
\hfil $\lambda$ & $\mu_{\lambda'}\ \text{(eigen value of}\ E_{\lambda'})$ 
& \text{$\dim\, U_{\lambda}$ in case $U_{\lambda}\not=\{0\}$} & $\text{genera for}\ U_{\lambda}\not=\{0\}$   \\
\hline
$[1^4]$ & $(2g-6)(2g-4)(2g-2)2g$ & $1$ & $g=4,5,\cdots$ \\
\hline
$[21^2]$ & $(2g-4)(2g-2)2g(2g+1)$ & $20$ & $g=3,4,5,\cdots$  \\
\hline
$[2^2]$ & $(2g-2)(2g-1)2g(2g+1)$ & $14$ & $g=2,3,4,5,\cdots$  \\
\hline
$[31]$ & $(2g-2)2g(2g+1)(2g+2)$ & $56$ & $g=2,3,4,5,\cdots$  \\
\hline
$[4]$ & $2g(2g+1)(2g+2)(2g+3)$ & $14$ & $g=1,2,3,4,5,\cdots$  \\                                                               
\hline
{} & {}  & $105$ & {}   \\
\noalign{\hrule height0.8pt}
\end{tabular}
\end{center}
\label{tab:8}
\end{table}

\begin{proof}[Proof of Theorem $\ref{thm:ortho}$] 
In view of Proposition \ref{prop:anti} and Proposition \ref{prop:lcdmu}, we have only to determine
the kernel of the {\it surjective} linear mapping
$$
\Phi: \Q\mathcal{D}^\ell (2k) \rightarrow (H_\Q^{\otimes 2k})^{\mathrm{Sp}}
$$
when we fix a specific genus $g$.
Let us denote the eigenvalue of $E_\lambda$ by $\mu_\lambda(g)$ to indicate its dependence
on the genus. The smallest factor in the expression of $\mu_\lambda(g)$,
considered as a polynomial in $g$ of degree $k$, is
$
(2g-2\lambda_1+2).
$
Therefore we have
$$
\mu_\lambda(g)\not=0\ \Leftrightarrow\ \lambda_1\leq g
$$
It follows that
$$
\mathrm{Ker}\,\Phi=\bigoplus_{|\lambda|=k, \lambda_1> g} E_\lambda
$$
and $\Phi$ induces an {\it isometric} linear isomorphism
$$
\bar{\Phi}:\bigoplus_{|\lambda|=k, \lambda_1\leq g} E_\lambda\cong (H_\Q^{\otimes 2k})^{\mathrm{Sp}}
$$
which is also an ``anti" -isomorphism as $\mathfrak{S}_{2k}$-modules.
Now for each Young diagram $\lambda$ with $|\lambda|=k$ and 
$h(\lambda)\leq g$, we set
$$
U_\lambda=\bar{\Phi}(E_{\lambda'})
$$
where $\lambda'$ denotes the conjugate Young diagram of $\lambda$ so that
$\lambda'_1=h(\lambda)\leq g$. Also as an $\mathfrak{S}_{2k}$-module,
we have isomorphisms
$$
U_\lambda\cong E_{\lambda'}\otimes [1^{2k}]_{\mathfrak{S}_{2k}}\cong 
(2\lambda')_{\mathfrak{S}_{2k}}\otimes [1^{2k}]_{\mathfrak{S}_{2k}}\cong 
\lambda^\delta_{\mathfrak{S}_{2k}}
$$
where the last isomorphism follows because it is easy to see that the identity
$$
(2\lambda')'=\lambda^\delta
$$
holds for any Young diagram $\lambda$. Thus we obtain an orthogonal direct sum decomposition
$$
(H_\Q^{\otimes 2k})^{\mathrm{Sp}}=\bigoplus_{|\lambda|=k, h(\lambda)\leq g} U_\lambda
$$
completing the proof.
\end{proof}

If we are concerned only with the {\it easier} part of Theorem \ref{thm:ortho}, namely
the irreducible decomposition of $(H_\Q^{\otimes 2k})^{\mathrm{Sp}}$ as an $\mathfrak{S}_{2k}$-module,
then there exists an alternative proof based on a classical result together with an argument
in the symplectic representation theory. We briefly sketch it here. First recall the following classical result.

\begin{thm}[\cite{fh},Theorem 6.3. (2)]
The $\mathrm{GL}(2g,\Q)$-irreducible decomposition of $H_\Q^{\otimes 2k}$ is give by
$$
H_\Q^{\otimes 2k}\cong \bigoplus_{|\lambda|=2k, h(\lambda)\leq 2g} 
\, \lambda_{\mathrm{GL}}^{\oplus \dim\, \lambda_{\mathfrak{S}_{2k}}}
$$
where $\lambda_{\mathrm{GL}}$ denotes the irreducible representation of $\mathrm{GL}(2g,\Q)$
corresponding to $\lambda$.
\label{thm:fh}
\end{thm}
If we combine this with the following fact
$$
\mathrm{dim}\, (\la_{\mathrm{GL}})^{\mathrm{Sp}}=
\begin{cases}
&1\quad (\text{$\la$: multiple double floors})\\
&0\quad (\text{otherwise})
\end{cases}
$$
which follows from the restriction law corresponding to
the pair $\mathrm{Sp}(2g,\Q)\subset \mathrm{GL}(2g,\Q)$
(see Proposition 4.1 of \cite{mss2}), then we obtain a proof of 
the easier part of Theorem \ref{thm:ortho}. We mention here that,
in a recent paper \cite{es},
Enomoto and Satoh describes the $\mathrm{Sp}$-irreducible
decomposition of $H_\Q^{\otimes 2k}$.


In a dual setting to the homomorphism
$\Phi: \Q\mathcal{D}^\ell (2k)\rightarrow (H_\Q^{\otimes 2k})^{\mathrm{Sp}}$
given in $\S$2, we define a linear mapping
$$
\mathcal{K}: H_\Q^{\otimes 2k}\rightarrow \Q\mathcal{D}^\ell (2k)
$$
by
$$
\mathcal{K}(\xi)=\sum_{C\in \mathcal{D}^\ell (2k)} \alpha_C(\xi)\, C\quad (\xi\in H_\Q^{\otimes 2k})
$$
where $\alpha_C\in \mathrm{Hom}(H_\Q^{\otimes 2k},\Q)^{\mathrm{Sp}}$ is defined as
$$
\alpha_C(u_1\otimes\cdots\otimes u_{2k})=\mathrm{sgn}\, C
\prod_{s=1}^k \mu(u_{i_s},u_{j_s})\ (u_i\in H_\Q).
$$
We use the same symbol $\mathcal{K}$ to denote its
restriction to the subspace $(H_\Q^{\otimes 2k})^{\mathrm{Sp}}$.

\begin{prop}
$\mathrm{(i)}\ \text{
The following diagram is commutative}$
$$
\begin{CD}
\Q\mathcal{D}^\ell (2k)   @>{\mathcal{M}}>> \Q\mathcal{D}^\ell (2k) \\
@V{\Phi}VV @|\\
(H_\Q^{\otimes 2k})^{\mathrm{Sp}} @>{\mathcal{K}}>> \Q\mathcal{D}^\ell (2k)
\end{CD}
$$
and the lower homomorphism $\mathcal{K}$ is 
injective for any $g$ and bijective for any $g\geq k$.


$\mathrm{(ii)\ }$
For any element $\xi\in E_\lambda\subset \Q\mathcal{D}^\ell(2k)$, we have the following equality 
$$
\mathcal{K}(\Phi(\xi))=\mu_{\lambda\,} \xi.
$$

\label{prop:dual}
\end{prop}

\begin{proof}
First we prove $\mathrm{(i)}$. The commutativity of the diagram follows by comparing the definitions
of the homomorphisms $\mathcal{M}, \Phi$ and $\mathcal{K}$. The latter claim follows from the basic 
result of Weyl that any $\mathrm{Sp}$-invariant can be described in terms of various contractions.

Next we prove $\mathrm{(ii)}$. Since $\xi$ is an eigenvector of $\mathcal{M}$
corresponding to the eigenvalue $\mu_\lambda$, we have $\mathcal{M}(\xi)=\mu_\lambda\, \xi$.
If we combine this with the commutativity just proved above, then the claim follows.
\end{proof}

\begin{proof}[Proof of Theorem $\ref{thm:aalpha}$]
The fact that $U_{\lambda}=\Phi(E_{\lambda'})$ follows from Proposition \ref{prop:anti}.
The last claim follows from $\mathrm{(ii)}$ of the above Proposition \ref{prop:dual}.
\end{proof}

\section{Proof of Theorem \ref{thm:lcd}}\label{sec:proof}

In this  section, we prove Theorem $\ref{thm:lcd}$. 
For that, we use a few terminologies about representations of the symmetric groups from \cite{fh}
which we recalled in $\S2$. 

Let 
$$
\hat{\ \ }: \Q\mathfrak{S}_{2k} \rightarrow \Q\mathfrak{S}_{2k} 
$$
denote the anti-involution induced by the anti-automorphism 
$\mathfrak{S}_{2k}\ni\gamma\mapsto\gamma^{-1}\in \mathfrak{S}_{2k}$.
\begin{lem}
$\mathrm{(i)\ }$ For any Young diagram $\lambda$, the following equalities hold
$$
\hat{a}_\lambda=a_\lambda,\quad \hat{b}_\lambda=b_\lambda,\quad \hat{c}_\lambda=c'_\lambda,\quad \hat{c'}_\lambda=c_\lambda.
$$

$\mathrm{(ii)\ }$ For any element $x\in \Q\mathfrak{S}_{2k}$ and $\xi, \eta\in\Q\mathcal{D}^\ell(2k)$, the following equality holds
$$
\langle \xi, x \eta\rangle=\langle \hat{x}\, \xi, \eta\rangle.
$$
\label{lem:hat}
\end{lem}

\begin{proof}
$\mathrm{(i)}$ follows from the definitions and  $\mathrm{(ii)}$
follows from the fact that the pairing $\langle\ ,\ \rangle$ is $\mathfrak{S}_{2k}$-equivariant.
\end{proof}

Now, for each Young diagram $\lambda$ with $k$ boxes, we have a map
$$
\left(\Q[\mathfrak{S}_{2k}]\, c'_{2\lambda} \right) \otimes \mathcal{D}^\ell(2k) \rightarrow \Q\mathcal{D}^\ell(2k).
$$
By Proposition \ref{prop:E} and the property of the Young symmetrizer,
the image of this map is precisely the subspace $E_\lambda\subset \Q\mathcal{D}^\ell(2k)$
which is the unique summand isomorphic to $(2\lambda)_{\mathfrak{S}_{2k}}$.

\begin{prop}
Two subspaces $E_\lambda$ and $E_\mu$ of $\Q\mathcal{D}^\ell(2k)$ are mutually orthogonal
to each other whenever $\lambda\not= \mu$. Namely we have
$$
\langle E_\lambda, E_\mu\rangle=\{0\}.
$$
\label{prop:lmortho}
\end{prop}

\begin{proof}
It suffices to show that, for any $C, D\in\mathcal{D}^\ell(2k)$
$$
\langle c_{2\lambda} C, c'_{2\mu} D\rangle=0
$$
because the elements of the form $c_{2\lambda} C$ generate the subspace $E_\lambda$ for any $\lambda$
and the same is true if we replace $c_{2\lambda} C$ by $c'_{2\lambda} C$.
On the other hand, by Lemma \ref{lem:hat} we have
\begin{align*}
\langle c_{2\lambda} C, c'_{2\mu} D\rangle&=\langle \hat{c'}_{2\mu} c_{2\lambda} C,  D\rangle\\
&=\langle {c}_{2\mu} c_{2\lambda} C,  D\rangle\\
&=0.
\end{align*}
The last equality follows because, as is well known, ${c}_{\mu} c_{\lambda}=0$
whenever $\mu\not=\lambda$.
This completes the proof.
\end{proof}

\begin{prop}
The subspace $E_\lambda\subset \Q\mathcal{D}^\ell(2k)$ is an eigenspace for the linear mapping 
$\mathcal{M}$ for any $\lambda$.
\label{prop:eigenspace}
\end{prop}

\begin{proof}
First we prove that $E_\lambda$ is an $\mathcal{M}$-invariant subspace, 
namely $\mathcal{M}(E_\lambda)\subset E_\lambda$. For this, it suffices to show that
$\mathcal{M}(E_\lambda)$ is orthogonal to $E_\mu$ for any $\mu\not=\lambda$.
This is because we already know that the whole space
$\Q\mathcal{D}^\ell(2k)$ can be written as
$$
\Q\mathcal{D}^{\ell}(2k)= \bigoplus_{|\lambda |=k}
E_{\lambda}\quad \text{(orthogonal direct sum)}
$$
(see Proposition \ref{prop:E}  and Proposition \ref{prop:lmortho}).
Recall that $\mathcal{M}(E_\lambda)$ is generated by the elements of the form
$
\sum_{D} \langle c_{2\lambda} C_1,D\rangle D
$
while $E_\mu$ is generated by the elements of the form
$
c'_{2\mu}\, C_2
$
where $C_1, C_2\in \mathcal{D}^{\ell}(2k)$.
Now we fix $C_1,C_2\in\mathcal{D}^{\ell}(2k)$ and compute
\begin{align*}
\left\langle \sum_{D} \langle c_{2\lambda} C_1,D\rangle D, c'_{2\mu}\, C_2\right\rangle &=
\sum_{D} \langle c_{2\lambda} C_1,D\rangle  \langle D, c'_{2\mu}\, C_2\rangle\\
&=\sum_{D}  \langle c_{2\lambda} C_1,D\rangle \langle c_{2\mu} D, C_2\rangle\\
&=\sum_{D}  \langle c_{2\lambda} C_1,c'_{2\mu} D\rangle \langle D, C_2\rangle\\
&=\sum_{D}  \langle c_{2\mu} c_{2\lambda} C_1, D\rangle \langle D, C_2\rangle\\
&=0
\end{align*}
because $c_{2\mu} c_{2\lambda}=0$. Therefore $\mathcal{M}(E_\lambda)\bot E_\mu$ as claimed.

Next observe that $\mathcal{M}$ is an $\mathfrak{S}_{2k}$-equivariant linear mapping because,
for any $\gamma\in \mathfrak{S}_{2k}$ and $C\in\mathcal{D}^{\ell}(2k)$, we have
\begin{align*}
\mathcal{M}(\gamma\, C)&=\sum_{D\in\mathcal{D}^{\ell}(2k)} \langle \gamma\, C,D\rangle D\\
&=\sum_{D\in\mathcal{D}^{\ell}(2k)} \langle C,\gamma^{-1}\, D\rangle D\\
&=\sum_{D\in\mathcal{D}^{\ell}(2k)} \langle C,D\rangle \gamma\, D\\
&=\gamma (\mathcal{M}(C)).
\end{align*}
It follows that any eigenspace of $\mathcal{M}$ is an $\mathfrak{S}_{2k}$-submodule.

Thus $E_\lambda$ is an $\mathcal{M}$-invariant and
$\mathfrak{S}_{2k}$-irreducible subspace of $\Q\mathcal{D}^{\ell}(2k)$. 
We can now conclude that it is an eigenspace as claimed.
\end{proof}

\begin{remark}
It follows from Theorem \ref{thm:aalpha} that, besides our inner product, the
subspaces $E_\lambda\subset \Q\mathcal{D}^\ell(2k)$ are mutually orthogonal to each other also with respect to
the {\it Euclidean metric} on $\Q\mathcal{D}^\ell(2k)$ induced by taking $\mathcal{D}^\ell(2k)$ as
an orthonormal basis. Indeed, this fact can be proved
by a standard argument in the theory of representations of symmetric groups
{\it without} mentioning our inner product on $\Q\mathcal{D}^\ell(2k)$.
The above argument is an adaptation of this to the case of our symmetric pairing
$\langle\ , \ \rangle$ which depends on $g$. It follows that the answer should be given by rescaling
eigenspaces by certain factors. Thus, in addition to the positivity, the essential point of our work is the
explicit determination of these factors as below.
\end{remark}

\begin{proof}[Proof of Theorem $\ref{thm:lcd}$]
In view of Proposition \ref{prop:E} and the arguments following it, the proof for the latter part
of this theorem is finished so that it remains to prove the former 
part.

We know already that each space $E_\lambda$ is an eigenspace (see Proposition \ref{prop:eigenspace}).
Therefore for any choice of a tableau on $2\lambda$ and for any $C\in \mathcal{D}^\ell(2k)$,
we should have an equality
$$
\mathcal{M}(c'_{2\lambda} C)=\sum_{D\in  \mathcal{D}^\ell(2k)} \langle c'_{2\lambda} C, D\rangle\, D
=\mu_\lambda c'_{2\lambda} C
$$
where $\mu_\lambda$ denotes the eigenvalue of $E_\lambda$. It follows that, for any $C, D$, we have
$$
\langle c'_{2\lambda} C, D\rangle=\text{$\mu_\lambda\cdot$ (coefficient of D in $c'_{2\lambda} C$)}.
$$
Our task is to determine $\mu_\lambda$ by the above equality and for that we may choose 
$C, D$ arbitrarily. We choose $C=D=C_0(2k)$ where $C_0(2k)$ denotes the following linear chord diagram
$$
C_0(2k)=\{\{1,2\},\cdots,\{2k-1,2k\}\}\in \mathcal{D}^\ell(2k).
$$
We set
$$
v_{2\lambda}(2k)=c'_{2\lambda} C_0(2k),\quad m_\lambda=\text{coefficient of $C_0(2k)$ in $v_{2\lambda}(2k)$}
$$
and we determine $\mu_\lambda$ from the identity
$$
\langle v_{2\lambda}(2k),C_0(2k)\rangle=m_\lambda \mu_{\lambda}.
$$

The rest of the proof is divided into three cases.

\vspace{3mm}
\noindent
$\mathrm{(I)}\ $ The case of $\lambda=[k]$.

In this case, the group $P_{2\lambda}$ is the whole group while $Q_{2\lambda}$ is the trivial
group. Hence, for any tableau $\tau$, we have
$$
v_{[2k]}=c_{[2k]}\, C_0(2k)=2^k k! \sum_{C\in \mathcal{D}^\ell(2k)} C
$$
so that $m_{[k]}=2^k k!$. Furthermore it was already proved in \cite{morita03} (proof of
Proposition 4.1) that
$$
\left\langle \sum_{C\in \mathcal{D}^\ell(2k)} C, C_0(2k)\right\rangle =2g(2g-2)\cdots(2g-2k+2).
$$
Therefore
$$
\mu_{[k]}=2g(2g-2)\cdots(2g-2k+2)
$$
as claimed.

\vspace{3mm}
\noindent
$\mathrm{(II)}\ $ The case of $\lambda=[1^k]$.

The case $k=1$ was already treated in $\mathrm{(I)}\ $ so that we assume $k\geq 2$ and compute
the value of
$$
\langle c'_{[2^k]} C_0, C_0\rangle
$$
step by step as follows. We have
$$
c'_{[2^k]}=b_{[2^k]}a_{[2^k]}
$$
and we choose a standard tableau $\tau$ on $[2^k]$ as below 
$$
\begin{bmatrix}
1 & 2\\
3 & 4\\
\cdots & \cdots\\
2k-1 & 2k
\end{bmatrix}.
$$
Then the element $b_{[2^k]}$ can be described as
$$
b_{[2^k]}=t_k t_{k-1}\cdots t_3 t_2
$$
where
\begin{align*}
t_2&=\{1-(13)\}\{1-(24)\},\ t_3=\{1-(15)-(35)\}\{1-(26)-(46)\},\cdots,\\
t_k&= \{1-(1,2k-1)-(3,2k-1)-\cdots -(2k-3,2k-1)\}\cdot\\
&\hspace{1cm} \{1-(2,2k)-(4,2k)-\cdots -(2k-2,2k)\}.
\end{align*}
We define
$$
c_1=a_{[2^k]},\quad c_s=t_s\cdots t_2 a_{[2^k]}\quad (s=2,\cdots,k)
$$ 
and compute
$$
\langle c_s C_0, C_0\rangle,\quad m_s=\text{coefficient of $C_0$ in $c_s C_0$}
$$
inductively. Here the essential point in our computation is Lemma \ref{lem:qij}. 
In fact, we use this lemma repeatedly in the whole of later argument and
we find a ``perfect canceling" in sgns of various operations on linear chord diagrams. 

In the first case $s=1$,
the group $P_{[2^k]}\cong \mathfrak{S}_2^k$ acts on $C_0$ trivially. Therefore
$$
c_1 C_0=a_{[2^k]} C_0=2^k C_0,\quad m_1=2^k
$$
and we have
$$
\langle c_1 C_0, C_0\rangle=m_1 (2g)^k.
$$
Next we compute the case $k=2$ 
$$
\langle c_2 C_0, C_0\rangle=2^k \langle t_1 C_0,C_0\rangle,\quad \text{$m_2$= coefficient of $C_0$ in $c_2 C_0$}.
$$
Among the $4$ elements in
$$
t_2=\{1-(13)\}\{1-(24)\}=1+(13)(24)-(13)-(24)
$$
exactly $2$ elements, namely $1, (13)(24)$ preserve $C_0$ unchanged.
Therefore $m_2=2 m_1$ and if we apply Lemma \ref{lem:qij}, we can conclude that
$$
\langle t_2 C_0,C_0\rangle=(2g)\{2\cdot 2g-2\cdot (-1)\}(2g)^{k-2}=(2g)2(2g+1) (2g)^{k-2}.
$$
Therefore we have
$$
\langle t_2 C_0,C_0\rangle=m_2 (2g)(2g+1) (2g)^{k-2}.
$$
Next we compute 
$$
\langle c_3 C_0, C_0\rangle=2^k \langle t_3 t_2 C_0,C_0\rangle,\quad \text{$m_3$= coefficient of $C_0$ in $c_3 C_0$}.
$$
Among the $9$ elements in
$$
t_3=\{1-(15)-(35)\}\{1-(26)-(46)\}
$$
exactly $3$ elements, namely $1, (15)(26), (35)(46)$ preserve $C_0$ unchanged.
Therefore $m_3=3 m_2$ and if we apply Lemma \ref{lem:qij}, we can conclude that
$$
\langle t_3 t_2 C_0,C_0\rangle=(2g)2(2g+1)\{3\cdot 2g-(9-3)\cdot(-1)\}(2g)^{k-3}=(2g)2(2g+1)3(2g+2) (2g)^{k-3}.
$$
Therefore we have
$$
\langle t_3 t_2 C_0,C_0\rangle=m_3 (2g)(2g+1) (2g+2)(2g)^{k-3}.
$$
Continuing the computation like this, we can prove that 
the equalities
\begin{align*}
\langle t_s\cdots t_1 C_0,C_0\rangle&=m_s (2g)(2g+1) (2g+2)\cdots (2g+s)(2g)^{k-s-1}\\
m_s&=2^k s!
\end{align*}
hold for all $s=1,\cdots,k$. 
We check the inductive argument from the $s$-th step to the next $(s+1)$-th step.

Among the $(s+1)^2$ elements in the expansion of
\begin{align*}
&\{1-(1,2s+1)-(3,2s+1)-\cdots -(2s-1,2s+1)\}\\
&\{1-(2,2s+2)-(4,2s+2)-\cdots -(2s,2s+2)\}
\end{align*}
exactly the following $(s+1)$ elements preserve $C_0$ unchanged, namely
$$
1, (1,2s+1)(2,2s+2),\cdots, (2s-1,2s+1)(2s,2s+2).
$$
Therefore $m_{s+1}=(s+1) m_{s}$ and if we apply Lemma \ref{lem:qij}, we can conclude that
\begin{align*}
&\langle t_{s+1} t_s\cdots t_2 C_0,C_0\rangle\\
&=m_s(2g)(2g+1)\cdots(2g+s-1)\{(s+1)\cdot 2g-((s+1)^2-(s+1))\cdot(-1)\}(2g)^{k-s}\\
&=(s+1)m_s 2g(2g+1)\cdots(2g+s) (2g)^{k-s}.
\end{align*}
Therefore we have
$$
\langle t_{s+1}\cdots t_2 C_0,C_0\rangle=m_{s+1} (2g)(2g+1)\cdots (2g+s)(2g)^{k-s-1}
$$
as required. 

\vspace{3mm}
\noindent
$\mathrm{(III)}\ $ The general case.

Finally we consider the general case. Roughly speaking, we combine the above two arguments, one is
the case $\mathrm{(I)}$ and the other is the case $\mathrm{(II)}$, to treat the general case.
Let $\lambda=[\lambda_1\lambda_2\cdots\lambda_h]$ be a general Young diagram with $|\lambda|=k$.
Choose a standard tableau on $2\lambda$, say $1,2,\cdots,2\lambda_1-1,2\lambda_1$ on the
first row, $2\lambda_1+1,2\lambda_1+2,\cdots,2\lambda_1+2\lambda_2-1,2\lambda_1+2\lambda_2$ on the
second row, and so on. Then we can write
$$
P_{2\lambda}=P_{2\lambda_1}\times P_{2\lambda_2}\times\cdots\times P_{2\lambda_h}
$$
where $P_{2\lambda_i}\ (i=1,\cdots,h)$ denotes the symmetric group consisting of permutations of 
labels attached to the boxes in the $i$-th row of $2\lambda$.
Thus we have
$$
a_{2\lambda}=a_{2\lambda_1}a_{2\lambda_2}\cdots a_{2\lambda_h}.
$$
If we apply the result of case $\mathrm{(I)}\ $ to each row of $2\lambda$,  then we obtain
\begin{align*}
\langle a_{2\lambda} C_0,C_0\rangle=&2^{\lambda_1}\lambda_1! \cdot 2g(2g-2)\cdots (2g-2\lambda_1+2)\cdot\\
&2^{\lambda_2}\lambda_2! \cdot 2g(2g-2)\cdots (2g-2\lambda_2+2)\cdot\\
&\cdots\\
&2^{\lambda_h}\lambda_h! \cdot 2g(2g-2)\cdots (2g-2\lambda_h+2)
\end{align*}
and
$$
\text{coefficient of $C_0$ in $a_{2\lambda} C_0\ = \prod_{i=1}^h 2^{\lambda_i}\lambda_i!=2^k \prod_{i=1}^h \lambda_i!$}.
$$

Next we consider the element $b_{2\lambda}$. The group $Q_{2\lambda}$ can be described as
$$
Q_{2\lambda}=Q_{2\lambda_1}\times Q_{2\lambda_1-1}\times\cdots\times Q_2\times Q_1
$$
where $Q_j\ (j=1,2,\cdots,2\lambda_1)$ denotes the symmetric group consisting of permutations of the
labels attached to the $j$-th column of $2\lambda$.
Here we mention that $2\lambda_1$ is the number of columns of $2\lambda$ and
also both of $Q_1$ and $Q_2$ are isomorphic to $\mathfrak{S}_{h}$.
Now we decompose the  element $b_{2\lambda}$ as
$$
b_{2\lambda}=b_{\lambda_1} \cdots b_2b_1
$$
where
$$
b_j=\prod_{\gamma\in Q_{2j-1}} (1+\mathrm{sgn}\,\gamma) \prod_{\delta\in Q_{2j}} (1+\mathrm{sgn}\,\delta). 
$$
First we consider the element $b_1a_{2\lambda} C_0$.
It can be seen that
we can apply the result of case $\mathrm{(II)}\ $ to the first two columns of $2\lambda$ from the left to obtain
\begin{align*}
\langle b_1 a_{2\lambda} C_0,C_0\rangle
=&\, h!\, 2^{\lambda_1}\lambda_1! \cdot 2g(2g-2)\cdots (2g-2\lambda_1+2)\cdot\\
&2^{\lambda_2}\lambda_2! \cdot [(2g+1)(2g-1)]\cdot (2g-4)\cdots (2g-2\lambda_2+2)\cdot\\
&\cdots\\
&2^{\lambda_h}\lambda_h! \cdot [(2g+h-1)(2g+h-3)]\cdot (2g-4)\cdots (2g-2\lambda_h+2)\\
=&\, h! \,2^k \prod_{i=1}^h \lambda_i!
\cdot 2g(2g-2)\cdots (2g-2\lambda_1+2)\cdot\\
& [(2g+1)(2g-1)]\cdot (2g-4)\cdots (2g-2\lambda_2+2)\cdot\\
&\cdots\\
& [(2g+h-1)(2g+h-3)]\cdot (2g-4)\cdots (2g-2\lambda_h+2)\\
=&\, h! \,2^k \prod_{i=1}^h \lambda_i!
\cdot [2g(2g+1)\cdots (2g+h-1)]\cdot
[(2g-2)(2g-1)\cdots (2g+h-3)]\cdot\\
&[(2g-4)(2g-6)]^{\lambda'_2}\cdots [(2g-2\lambda_1+4)(2g-2\lambda_1+2)]^{\lambda'_{\lambda_1}}
\end{align*}
and
$$
\text{coefficient of $C_0$ in $b_1a_{2\lambda} C_0\ = h! \,2^k \prod_{i=1}^h \lambda_i!$}
$$
where $\lambda'=[h\lambda'_2\cdots \lambda'_{\lambda_1}]$ denotes the conjugate Young diagram of $\lambda$.

Next we can apply the result of case $\mathrm{(II)}\ $ to the two columns,
the third and the fourth from the left of $2\lambda$, to obtain
\begin{align*}
\langle b_2 b_1 a_{2\lambda} C_0,C_0\rangle=&h! \lambda'_2! \,2^k \prod_{i=1}^h \lambda_i! 
\cdot 2g(2g-2)\cdots (2g-2\lambda_1+2)\cdot\\
& (2g+1)(2g-1)\cdot [(2g-3)(2g-5)]\cdot (2g-8)\cdots (2g-2\lambda_2+2)\cdot\\
&\cdots\\
& (2g+h-1)(2g+h-3)\cdot [(2g+h-5)(2g+h-7)]\cdot(2g-8)\cdots (2g-2\lambda_h+2)\\
=&\, h! \lambda'_2!\,2^k \prod_{i=1}^h \lambda_i!
\cdot [2g(2g+1)\cdots (2g+h-1)]\cdot
[(2g-2)(2g-1)\cdots (2g+h-3)]\cdot\\
& \cdot [(2g-4)(2g-3)\cdots (2g+\lambda'_2-1)]\cdot
[(2g-6)(2g-5)\cdots (2g+\lambda'_2-3)]\cdot\\
&[(2g-8)(2g-10)]^{\lambda'_3}\cdots [(2g-2\lambda_1+4)(2g-2\lambda_1+2)]^{\lambda'_{\lambda_1}}
\end{align*}
and
$$
\text{coefficient of $C_0$ in $b_2 b_1a_{2\lambda} C_0\ = h!\lambda'_2! \,2^k \prod_{i=1}^h \lambda_i!$}.
$$

Continuing in this way until the last two columns from the right of $\lambda$,
we finally obtain
\begin{align*}
\langle c'_{2\lambda} C_0,C_0\rangle=& m_{2\lambda}\cdot 2g(2g-2)\cdots (2g-2\lambda_1+2)\cdot\\
& (2g+1)(2g-1)\cdots (2g-2\lambda_2+3)\cdot\\
&\cdots\\
&(2g+h-1)(2g+h-3)\cdots (2g+h-2\lambda_h+1)\\
m_{2\lambda}=&\prod_{i=1}^h 2^{\lambda_i}\lambda_i! \prod_{j=1}^{\lambda_1}  \lambda'_j!=2^k \prod_{i=1}^h \lambda_i! \prod_{j=1}^{\lambda_1}  \lambda'_j!.
\end{align*}

Here we would like to emphasize again that an essential use of Lemma \ref{lem:qij} is made in the above computation. 
This completes the proof.
\end{proof}

\section{Topological study of the tautological algebra of the moduli space of curves}

As already mentioned in the introduction, this work was originally motivated by 
a topological study of the tautological algebra of the moduli space of curves.

Let $\mathcal{M}_g$ be the mapping class group of $\Sigma_g$
and let $\mathcal{R}^*(\mathcal{M}_g)$ denote the tautological algebra {\it in cohomology} of 
$\mathcal{M}_g$ . Namely it is the subalgebra of $H^*(\mathcal{M}_g;\Q)$
generated by the Mumford-Morita-Miller tautological classes $e_i\in H^{2i}(\mathcal{M}_g;\Q)$.
We also consider the mapping class group $\mathcal{M}_{g,*}$ of 
$\Sigma_g$ relative to the base point $*\in\Sigma_g$ and the corresponding 
tautological algebra $\mathcal{R}^*(\mathcal{M}_{g,*})$. In this case,
it is the subalgebra of $H^*(\mathcal{M}_{g,*};\Q)$ generated by 
$e_i$ and the Euler class $e\in H^2(\mathcal{M}_{g,*};\Q)$.

In the context of algebraic geometry, there is a more refined form of the tautological algebra. 
Namely, if we denote by $\mathbf{M}_g$ the moduli space of curves of genus $g$, then 
its tautological algebra, denoted by $\mathcal{R}^*(\mathbf{M}_g)$, is defined to be the
subalgebra of the Chow algebra $\mathcal{A}^*(\mathbf{M}_g)$ of $\mathbf{M}_g$ generated
by the tautological classes $\kappa_i\in \mathcal{A}^i(\mathbf{M}_g)$ introduced by Mumford.
There is a natural surjection $\mathcal{R}^*(\mathbf{M}_g)\rightarrow \mathcal{R}^*(\mathcal{M}_g)$
which sends $\kappa_i$ to $(-1)^{i+1} e_i$.

There have been obtained many interesting results about the tautological algebras,
Faber's beautiful conjecture given in \cite{faber} together with his related works
being a strong motivation for them.
In particular, Faber and Zagier proposed a set of conjectural relations in
$\mathcal{R}^*(\mathbf{M}_g)$ and Faber in fact proved that it
gives an affirmative answer to his conjecture including the
most difficult part, namely the perfect pairing property, up to genus $g\leq 23$.
Very recently, Pandharipande and Pixton \cite{pp} proved a remarkable
result that the Faber-Zagier relations are actual relations in $\mathcal{R}^*(\mathbf{M}_g)$ for all $g$.
However the structure of $\mathcal{R}^*(\mathbf{M}_g)$ for general $g$ remains mysterious.
For more details of these developments, we refer to Faber's recent survey paper 
\cite{faberpcmi} as well as the references therein.
The main tools in these works belong to algebraic geometry which are quite powerful.
A theorem of Looijenga given in \cite{looijenga}, which is one of the earliest
results in this area, is a typical example. 
See also recent works \cite{yin}\cite{yin2} of 
Yin for the case of $\mathcal{R}^*(\mathbf{M}_{g,1})$.

On the other hand, there are a few topological approaches to this subject.
The first one was given in our papers \cite{moritaj1}\cite{moritaj2}.
Recently Randal-Williams \cite{rw} greatly generalized this 
to obtain very interesting results.  See also 
Grigoriev \cite{grigoriev} for the case of higher dimensional manifold bundles.

In \cite{morita03} we adopted a closely related but a bit different 
approach, namely a combination of topology and representation theory. 
In the next section we extend this approach and
give a systematic way of producing relations among
the $\mathrm{MMM}$ tautological classes in $\mathcal{R}^*(\mathcal{M}_g)$ and $\mathcal{R}^*(\mathcal{M}_{g,*})$.
For that, in the remaining part of this section, we summarize former results.

By analyzing the natural action of $\mathcal{M}_g$ on the {\it third} nilpotent quotient of
$\pi_1\Sigma_g$,
in \cite{morita01} we constructed the following commutative diagram
\begin{equation}
\begin{CD}
\pi_1\Sigma_g @>>> [1^2]_{\mathrm{Sp}}\tilde{\times} H_\Q \\
@VVV @VVV\\
\mathcal{M}_{g,*}   @>{\rho_2}>> (([1^2]^{\text{torelli}}_{\mathrm{Sp}}\oplus [2^2]_{\mathrm{Sp}})\tilde{\times} \wedge^3 H_\Q)\rtimes
\mathrm{Sp}(2g,\Q) \\
@VV{p}V @VVV \\
\mathcal{M}_g @>{\bar{\rho}_2}>> ([2^2]_{\mathrm{Sp}}\tilde{\times} U_\Q)\rtimes
\mathrm{Sp}(2g,\Q).
\end{CD}
\label{eq:rho2}
\end{equation}
This extends our earlier work in \cite{morita96} in which we treated the case of the action on the second nilpotent quotient.
Here $U_\Q=\wedge^3H_\Q/H_\Q$ and $[2^2]_{\mathrm{Sp}}$ is the unique $\mathrm{Sp}$-representation 
corresponding to the Young diagram $[2^2]$ which appears in both of $\wedge^2 U^*_\Q$ 
and $\wedge^2(\wedge^3 H^*_\Q)$.
Also $([2^2]_{\mathrm{Sp}})$ denotes the ideal generated by $[2^2]_{\mathrm{Sp}}$.
$(([1^2]^{\text{torelli}}_{\mathrm{Sp}}\oplus [2^2]_{\mathrm{Sp}})$ is defined similarly
(see \cite{morita01} for details). 

In our joint works with Kawazumi \cite{kam}\cite{kamp}, we determined the homomorphisms
$\rho^*_2, \bar{\rho}^*_2$ in cohomology which are induced by the representations
$\rho_2,\bar{\rho}_2$ in the above diagram \eqref{eq:rho2}. More precisely, we proved that the images of them
are precisely the tautological algebras and also that these representations are isomorphisms in the
stable range of Harer \cite{harer} (so that the images are in fact isomorphic to
the whole rational cohomology algebras of the mapping class groups
by the definitive result of Madsen and Weiss \cite{mw}). 
In the proof of these results, the work of Kawazumi in \cite{kaw}
played a crucial role.
Thus we have the following commutative diagram

\begin{equation}
\begin{CD}
 \left(\wedge^*(\wedge^3 H^*_\Q)/([1^2]^{\text{torelli}}_{\mathrm{Sp}}\oplus [2^2]_{\mathrm{Sp}})\right)^{\mathrm{Sp}}
@>{\rho_2^*}>{\text{stably $\cong$}}> \mathcal{R}^*(\mathcal{M}_{g,*}) \\
@AAA @AA{p^*}A\\
\left(\wedge^*U^*_\Q/([2^2]_{\mathrm{Sp}})\right)^{\mathrm{Sp}}
@>{\bar{\rho}_2^*}>{\text{stably $\cong$}}> \mathcal{R}^*(\mathcal{M}_{g}) .
\end{CD}
\label{eq:rho2*}
\end{equation}
Here $([2^2]_{\mathrm{Sp}})$ denotes the ideal generated by $[2^2]_{\mathrm{Sp}}$.
$(([1^2]^{\text{torelli}}_{\mathrm{Sp}}\oplus [2^2]_{\mathrm{Sp}})$ is defined similarly.
Furthermore we have given explicit formulas for the homomorphisms $\rho_2^*, \bar{\rho}_2^*$
in terms of graphs which we now recall.

Let $\mathcal{G}$ be the set of isomorphism classes of {\it connected} trivalent graphs
and let $\Q[\Gamma;\Gamma\in\mathcal{G}]$ denote the polynomial algebra generated by
$\mathcal{G}$. We define the degree of a graph $\Gamma\in\mathcal{G}$ to be the number
of vertices of it. Thus the subspace $\Q[\Gamma;\Gamma\in\mathcal{G}]^{(2k)}$ consisting
of homogeneous polynomials of degree $2k$ can be naturally identified with the linear space
spanned by trivalent graphs with $2k$ vertices. Also let $\mathcal{G}^0$
be the subset of $\mathcal{G}$ consisting of connected trivalent graphs
{\it without loops}. 
In \cite{morita96}, the following commutative diagram
was constructed  
\begin{equation}
\begin{CD}
 \Q[\Gamma;\Gamma\in\mathcal{G}]
@>>> (\wedge^*(\wedge^3 H^*_\Q))^{\mathrm{Sp}} \\
@A{\iota_g}AA @AAA\\
\Q[\Gamma;\Gamma\in\mathcal{G}^0]
@>>> (\wedge^*U^*_\Q)^{\mathrm{Sp}}
\end{CD}
\label{eq:graph}
\end{equation}
where both of the horizontal homomorphisms are surjective
and bijective in a certain stable range and $\iota_g$ denotes some
{\it twisted} inclusion. 

Then in \cite{kamp}, the following commutative diagram
$$
\begin{CD}
 \Q[\Gamma;\Gamma\in\mathcal{G}]
@>{\alpha}>> \Q[g][e,e_1,e_2,\cdots] \\
@A{\iota_g}AA @AA{\cup}A\\
\Q[\Gamma;\Gamma\in\mathcal{G}^0]
@>{\beta}>> \Q[g][e_1,e_2,\cdots]
\end{CD}
$$
was constructed which realizes the composition of the commutative diagrams \eqref{eq:graph} followed by \eqref{eq:rho2*}
by means of explicit formulas $\alpha$ and $\beta$. In particular, for each trivalent graph 
$\Gamma$ with $2k$ vertices, we have the associated characteristic class
$$
\alpha_\Gamma\in \Q[g][e,e_1,e_2,\ldots]^{(2k)}.
$$
If we fix a genus $g$, then 
the commutative diagram \eqref{eq:rho2*} is isomorphic to the following
$$
\begin{CD}
 \Q[\Gamma;\Gamma\in\mathcal{G}]/\widetilde{W}
@>{\alpha_g}>{\text{stably $\cong$}}> \Q[e,e_1,e_2,\cdots] \\
@A{\iota_g}AA @AA{\cup}A\\
\Q[\Gamma;\Gamma\in\mathcal{G}^0]/W
@>{\beta_g}>{\text{stably $\cong$}}> \Q[e_1,e_2,\cdots]
\end{CD}
$$
where $\widetilde{W}$ denotes certain equivalence relation in $\Q[\Gamma;\Gamma\in\mathcal{G}]$
induced by graphical operations on trivalent graphs including the Whitehead move 
(or the $\mathrm{IH}$ move) and $W$ is a similar one.
We mention that the relation between the ideal $([2^2]_{\mathrm{Sp}})$ and
the $\mathrm{IH}$ move for the lower case 
was first studied by Garoufalidis and Nakamura in \cite{gn}. 

\section{A set of relations in $\mathcal{R}^*(\mathcal{M}_g)$ and $\mathcal{R}^*(\mathcal{M}_{g,*})$}

In this section, we apply Theorem \ref{thm:ortho} to the description of the tautological algebras
$\mathcal{R}^*(\mathcal{M}_g)$ and $\mathcal{R}^*(\mathcal{M}_{g,*})$ in terms of certain symplectic invariant tensors
as well as trivalent graphs,
which we recall in the previous section, and obtain relations among the $\mathrm{MMM}$ tautological classes
in a systematic way. This extends our former work in \cite{morita03}.

The symplectic invariant tensors we are concerned is $\left(\wedge^{2k}(\wedge^3 H_\Q)\right)^{\mathrm{Sp}}$
which can be considered as both a natural quotient space as well as the
subspace  of $(H_\Q^{\otimes 6k})^{\mathrm{Sp}}$. Furthermore this subspace
can be described by means of the natural action of $\mathfrak{S}_{6k}$ on the 
whole space $(H_\Q^{\otimes 6k})^{\mathrm{Sp}}$. It follows that the orthogonal
direct sum decomposition given by Theorem \ref{thm:ortho} induces that
of the space $\left(\wedge^{2k}(\wedge^3 H_\Q)\right)^{\mathrm{Sp}}$
so that we obtain the following commutative diagram
\begin{equation}
\begin{CD}
(H_\Q^{\otimes 6k})^{\mathrm{Sp}}
@>{\text{quotient, contain}}>{\supset}> \left(\wedge^{2k}(\wedge^3 H_\Q)\right)^{\mathrm{Sp}}\\
@| @|\\
\bigoplus_{|\lambda|=3k} U_\lambda
@>{\text{quotient, contain}}>{\supset}> \bigoplus_{|\lambda|=3k} L_\lambda
\end{CD}
\label{eq:qs}
\end{equation}
where $L_\lambda$ denotes certain subspace of $U_\lambda$. 
The totality of the upper quotient homomorphism 
in this diagram for all $k$ can be realized graphically by the following collapsing operation.
\begin{definition}[\cite{morita03}]
For each $C\in\mathcal{D}^\ell(6k)$, let $\Gamma_C$ denote the trivalent  graph defined as follows.
We understand $C$ as a disjoint union of $3k$ chords each of which connects
the two points $i_s, j_s\ (s=1,2,\cdots,3k)$ in the set of vertices $\{1,2,\cdots,6k\}$.
Then $\Gamma_C$ is the graph obtained from $C$ by joining the three vertices
in each of the sets $\{1,2,3\},$ $\{4,5,6\},$ $\cdots,$ $\{6k-2,6k-1,6k\}$ to a single point.
Thus $\Gamma_C$ is a (possibly disconnected) trivalent graph with $2k$ vertices.
By extending this operation linearly, we obtain a linear mapping
$$
\Q\mathcal{D}^\ell(6k)\rightarrow \Q[\Gamma;\Gamma\in\mathcal{G}]^{(2k)}.
$$
For each element $\xi\in \Q\mathcal{D}^\ell(6k)$, we denote its image 
under the above mapping by $\Gamma_\xi$.
\label{def:g1}
\end{definition}

Then we have the following commutative diagram
\begin{equation}
\begin{CD}
\Q\mathcal{D}^\ell(6k)
@>{\text{collapse}}>{\supset}> \Q[\Gamma;\Gamma\in\mathcal{G}]^{(2k)}\\
@| @|\\
\bigoplus_{|\lambda|=3k} E_\lambda
@>{\text{collapse}}>{\supset}> \bigoplus_{|\lambda|=3k} L'_\lambda
\end{CD}
\label{eq:ef}
\end{equation}
where $L'_\lambda$ denotes the subspace corresponding to $L_\lambda$
(and in fact isomorphic to $L_{\lambda'}$).
In view of \eqref{eq:graph}, this is equivalent to the diagram \eqref{eq:qs}.

We need one more terminology to formulate our result and that concerns the 
following direct sum decomposition
\begin{equation}
(\wedge^{2k}(\wedge^3 H_\Q))^{\mathrm{Sp}}
\cong \bigoplus_{p=0}^{2k} (\wedge^{2k-p} U_\Q\otimes \wedge^p H_\Q)^{\mathrm{Sp}}.
\label{eq:up}
\end{equation}
This is based on the fact that $\wedge ^3 H_\Q$ is not irreducible as an $\mathrm{Sp}$-module
but is a sum of two irreducible summands $\wedge^3 H_\Q\cong U_\Q\oplus H_\Q$.
Notice that the above decomposition is an unstable operation in the sense that it depends on the
genus $g$ and also that it cannot be described in terms of the action of the symmetric group.

To realize the above decomposition \eqref{eq:up} graphically,
we recall a few terminologies from our paper \cite{morita03} in suitably modified forms.
Consider the mapping
\begin{equation}
\begin{split}
\wedge^{2k} &(\wedge^3 H_\Q)\ni 
\xi=\xi_1\wedge\cdots\wedge\xi_{2k}\\
&\mapsto
\xi^{(p)}=\sum_{i_1<\cdots <i_p} \xi_1\wedge\cdots\wedge \widetilde{C}\xi_{i_1}\wedge\cdots
\wedge \widetilde{C}\xi_{i_p}\wedge\cdots\wedge \xi_{2k}\in \wedge^{2k}(\wedge^3 H_\Q) 
\end{split}
\label{eq:xip}
\end{equation}
where $\xi_i\in \wedge^3 H_\Q$, $0\leq p\leq 2k$ and $\widetilde{C}:\wedge^3 H_\Q\rightarrow \wedge^3 H_\Q$
is the mapping defined by
$$
\widetilde{C}(u\wedge v \wedge w)=\frac{1}{g-1}\, (\mu(u,v)w+\mu(v,w)u+\mu(w,u)v)\wedge\omega_0.
$$

Now let $\Gamma$ be a trivalent graph with $2k$ vertices  and let $V_\Gamma$ denote the 
set of vertices. Fix a subset $F\subset V_\Gamma$ consisting of $p$ vertices where $0\leq p\leq 2k$.
We denote by $\mathcal{P}(F)$ the set of all ways of choices $\varpi$ where, for each vertex
$v\in F$, $\varpi(v)$ specifies $2$ edges emerging from $v$ out of the $3$ such edges.
Since there are $3$ edges emerging from each vertex, $\mathcal{P}(F)$
has $3^p\binom{2k}{p}$ elements. 

For each $\varpi\in\mathcal{P}(F)$, we define a trivalent graph $\Gamma_\varpi$ and a number
$\ell(\varpi)$ as follows. At each vertex $v\in F$, cut the $2$ edges emerging from it which correspond to $\varpi(v)$
at the one third point along each edge from $v$. Then we connect the two endpoints of the shortened edges emerging
from $v$ to make a loop. At the same time, we also connect the two endpoints of the remaining 
$2$ edges. After performing this operation at every vertex $v\in F$, we are left with a trivalent graph
together with certain number of disjoint circles and/or intervals. We define $\Gamma_\varpi$ to be this
trivalent graph throwing away the other parts. The number $\ell(\varpi)$ is defined to be that of the
disjoint circles. Namely it is the sum of the number of loops and that of double edges ``occurring" 
in $\varpi$.

\begin{definition}[{\it graphical contraction}]
Let $\Gamma$ be a trivalent graph with $2k$ vertices. For each $0\leq p \leq 2k$, we define 
$$
\Gamma^{(p)}=\frac{1}{(2g-2)^p}\sum_{F\subset V_\Gamma} 
\sum_{\varpi\in\mathcal{P}(F)} (-1)^p(-2g)^{\ell(\varpi)} {\Gamma}_\varpi
\ \in \Q[\Gamma;\Gamma\in\mathcal{G}]^{(2k)}
$$
where $F$ runs through all subsets of $V_\Gamma$ consisting of $p$ elements. 
Furthermore we set
$$
\bar{\Gamma}=\Gamma-\Gamma^{(1)}+\Gamma^{(2)}-\cdots+\Gamma^{(2k)}.
$$
Finally we extend these operations on the whole space $\Q[\Gamma;\Gamma\in\mathcal{G}]^{(2k)}$ by
linearity. 
\label{def:g2}
\end{definition}

In summary, each element $\xi\in\Q\mathcal{D}^\ell(6k)$ gives rise to the associated element
$\Gamma_\xi\in \Q[\Gamma;\Gamma\in\mathcal{G}]^{(2k)}$ by Definition \ref{def:g1}
and then we have the following series of elements
$$
\Gamma_\xi=\Gamma_\xi^{(0)}, \Gamma_\xi^{(1)},\cdots,\Gamma_\xi^{(2k)},\bar{\Gamma}_\xi
$$
belonging to $\Q[\Gamma;\Gamma\in\mathcal{G}]^{(2k)}$ by Definition \ref{def:g2}.
Then it was shown in \cite{morita03} that the above graphical operations realize the decomposition
\eqref{eq:up}.

Now we are ready to state our main result in this section.
\begin{thm}
$\mathrm{(i)\ }$ In the tautological algebra $\mathcal{R}^*(\mathcal{M}_{g})$, the following relations hold.
$$
\text{For all $\xi\in E_\lambda \subset \Q\mathcal{D}^\ell(6k)$ with $\lambda_1>g$, 
we have $\alpha_{\bar{\Gamma}_\xi}=0\in\mathcal{R}^{2k}(\mathcal{M}_{g})$}.
$$

$\mathrm{(ii)\ }$ In the tautological algebra $\mathcal{R}^*(\mathcal{M}_{g,*})$, the following relations hold.
$$
\text{For all $\xi\in E_\lambda \subset \Q\mathcal{D}^\ell(6k)$ with $\lambda_1>g$, 
we have $\alpha_{\Gamma^{(p)}_\xi}=0\in\mathcal{R}^{2k}(\mathcal{M}_{g,*})$ for any $0\leq p\leq 2k$}.
$$
\label{thm:rel}
\end{thm}

\begin{proof}
We first prove $\mathrm{(ii)}$.
Let $\xi\in E_\lambda\subset \Q\mathcal{D}^\ell(6k)$ be any element. Then we have $\Phi(\xi)\in U_{\lambda'}$ 
by Theorem \ref{thm:aalpha}. If we further assume that $\lambda_1>g$, then $h(\lambda')=\lambda_1>g$
so that $\Phi(\xi)=0\in (H_\Q^{\otimes 6k})^{\mathrm{Sp}}$ by Theorem \ref{thm:ortho}.
It follows that the projected image of $\Phi(\xi)$ in $(\wedge^{2k}(\wedge^3 H_\Q))^{\mathrm{Sp}}$
is also trivial. Furthermore all the components of this element with respect to the direct sum
decomposition \eqref{eq:up} are trivial because they are obtained by applying the
operations \eqref{eq:xip}. Since these tensorial operations are realized by the graphical
operations given in Definition \ref{def:g1} and Definition \ref{def:g2}, the claim holds.

Next we prove $\mathrm{(i)}$. The graphical operation
$$
\Gamma\mapsto \bar{\Gamma}=\Gamma-\Gamma^{(1)}+\Gamma^{(2)}-\cdots+\Gamma^{(2k)}
$$
given in Definition \ref{def:g2} realizes the tensorial operation of taking the
$(\wedge^{2k} U_\Q)^{\mathrm{Sp}}$ part out of the whole $(\wedge^{2k}(\wedge^3 H_\Q))^{\mathrm{Sp}}$
in the direct sum decomposition \eqref{eq:up}.
On the other hand, the
diagram \eqref{eq:rho2*} shows that this operation
extracts exactly the 
``basic" part with respect to the projection $\mathcal{M}_{g,*}\rightarrow \mathcal{M}_{g}$,
namely the part which comes from the "base" $\mathcal{M}_{g}$ 
(observe here that we have a natural
injection $\mathcal{R}^*(\mathcal{M}_g) \subset \mathcal{R}^*(\mathcal{M}_{g,*})$).
The claim follows.
\end{proof}

\begin{remark}
The first case $\lambda=[3k]$ in the above theorem was treated in our former paper \cite{morita03} completely.
In this case $2\lambda=[6k]$ is the one dimensional trivial representation of $\mathfrak{S}_{6k}$
so that $\mathrm{dim}\, E_\lambda=1$. We have determined the relations 
$\alpha_{\bar{\Gamma}_\xi}=0\in\mathcal{R}^{2k}(\mathcal{M}_{g})$
explicitly  for all $k$ and $g=3k-1,3k-2,\ldots,2$, where $\xi$ denotes a non-trivial
unique (up to scalars) element of $E_\lambda$. In this computation, the eigenvalue 
$$
\mu_{[3k]}=2g(2g-2)\cdots (2g-6k+2)
$$
of $E_{[3k]}$, which is the smallest one, played an important role.

As an application of this first case, we showed that $\mathcal{R}^*(\mathcal{M}_g)$
is generated by the first $[g/3]$ $\mathrm{MMM}$ classes. Later Ionel \cite{ionel} proved a stronger result that this fact holds
at the Chow algebra level, namely in $\mathcal{R}^*(\mathbf{M}_g)$.

The second case corresponds to $\lambda=[3k-1,1]$ and $\mathrm{dim}\, E_\lambda=9k(2k-1)$,
the third case corresponds to $\lambda=[3k-2,2]\ (k\geq 2)$ and $\mathrm{dim}\, E_\lambda=\frac{1}{2}k(3k-1)(6k-1)(6k-7)$,
and so on. Schematically, we may write
\begin{align*}
&[3k]\mapsto 0\quad (g\leq 3k-1)\ \Rightarrow\ \text{$\mathcal{R}^*(\mathcal{M}_g)$ is generated by the first $[g/3]$ $\mathrm{MMM}$ classes}\\
& [3k-1,1]\mapsto 0\quad (g\leq 3k-2)\\
& [3k-2,2], [3k-2,1^2] \mapsto 0\quad (g\leq 3k-3)\\
& [3k-3,3], [3k-3,21], [3k-3,1^3] \mapsto 0\quad (g\leq 3k-4)\\
& [3k-4,4], [3k-4,31], [3k-4,2^2], [3k-4,21^2], [3k-3,1^4] \mapsto 0\quad (g\leq 3k-5)\\
& \cdots
\end{align*}
Although we have to consider the collapsing $E_\lambda\rightarrow L_{\lambda'}$ described
in diagram \eqref{eq:ef}, it seems reasonable to expect that many relations can be obtained in this way.
We would like to pursue this in future. 
\end{remark}

In view of the fact that Theorem \ref{thm:ortho} gives a complete description how the 
$\mathrm{Sp}$-invariant tensors degenerate according as the genus decreases from the stable range to the
last case $g=1$, it would be reasonable to make the following.

\begin{conj}
The relations given in Theorem \ref{thm:rel} make a complete set of relations for the tautological algebras
$\mathcal{R}^*(\mathcal{M}_{g})$ and $\mathcal{R}^*(\mathcal{M}_{g,*})$.
\label{conj:complete}
\end{conj}


\section{Concluding remarks}

In a joint work  \cite{mss4} with T. Sakasai and M. Suzuki, another
role of the canonical metric is given. In this work, we investigate the structure of the
symplectic derivation Lie algebra $\mathfrak{h}_{g,1}$ which is the Lie algebra consisting of symplectic
derivations of the free Lie algebra generated by $H_\Q$. This Lie algebra is a very important subject in topology 
and also in other related fields of mathematics including number theory. 
In particular, the symplectic invariant part $\mathfrak{h}_{g,1}^{\mathrm{Sp}}$
of this Lie algebra concerns many deep questions in various contexts, e.g. homology cobordism invariants for homology cylinders,
the rational cohomology group of the outer automorphism group of free groups and the Galois obstructions
in the theory of arithmetic mapping class groups.
We obtain several results about $\mathfrak{h}_{g,1}^{\mathrm{Sp}}$
and we refer to the above cited paper for details.

Finally we would like to mention another possible application of our metric. 
That is concerned with ``harmonic representatives" of cohomology classes
which arise in various symplectic Lie algebras, by which we mean graded Lie algebras
such that each graded piece is an $\mathrm{Sp}$-module and the bracket operation
is an $\mathrm{Sp}$-morphism.
We are keeping in mind the Lie algebra $\mathfrak{h}_{g,1}$ mentioned above
and more generally Kontsevich's graph (co)homology theory (see \cite{kontsevich1}\cite{kontsevich2}).
It might be worthwhile to investigate whether harmonic representatives of 
cohomology classes of these Lie algebras would uncover some unknown geometric property of them.

In particular, the case of the Gel'fand-Fuks cohomology of formal Hamiltonian vector fields
on $H_\R=H_\Q\otimes\R$, both in the stable as well as the unstable contexts, seems 
to be interesting. In the simplest case of formal Hamiltonian vector fields on the plane,
Gel'fand, Kalinin and Fuks found in \cite{gkf} an {\it exotic class} and Metoki 
found another one in \cite{metoki}. The problem of determining whether these classes are non-trivial
as characteristic classes of transversely symplectic foliations remains unanswered up to the present.
In the framework due to Kontsevich \cite{kontsevich3}, it was shown in our paper with Kotschick \cite{km}
that the Gel'fand-Kalinin-Fuks class can be decomposed into a product of certain leaf cohomology class
and the transverse symplectic form and in \cite{mikami} Mikami proved a similar statement for the
Metoki class. It would be worthwhile to study whether the harmonic representatives of these leaf cohomology
classes would shed light on the above difficult problem.

\bibliographystyle{amsplain}

\end{document}